\documentclass[oneside,english,smallextended]{amsart}
\usepackage[T1]{fontenc}
\usepackage[latin9]{inputenc}
\usepackage{mathrsfs}
\usepackage{amsthm}
\usepackage{amsbsy}
\usepackage{amssymb}
\usepackage{setspace}
\usepackage{esint}
\usepackage[all]{xy}
\makeatletter
\numberwithin{equation}{section}
\numberwithin{figure}{section}
\theoremstyle{plain}
\newtheorem{thm}{\protect\theoremname}
  \theoremstyle{plain}
  \newtheorem{cor}[thm]{\protect\corollaryname}
  \theoremstyle{plain}
  \newtheorem{lem}[thm]{\protect\lemmaname}
  \theoremstyle{plain}
  \newtheorem{prop}[thm]{\protect\propositionname}
  \theoremstyle{remark}
  \newtheorem{claim}[thm]{\protect\claimname}

\makeatother

\usepackage{babel}
  \providecommand{\claimname}{Claim}
  \providecommand{\corollaryname}{Corollary}
  \providecommand{\lemmaname}{Lemma}
  \providecommand{\propositionname}{Proposition}
\providecommand{\theoremname}{Theorem}

\begin{document}

\title{Structure of the Unramified $L$-packet }

\author{Manish Mishra}

\email{mmishra@math.huji.ac.il}

\address{Einstein Institute of Mathematics, The Hebrew University of Jerusalem,
Jerusalem, 91904, Israel}
\begin{abstract}
Let $\boldsymbol{G}$ be an unramified connected reductive group defined
over a non-archemedian local field $k$ and let $\boldsymbol{T}$
be a maximal torus in $\boldsymbol{G.}$ Let $\lambda$ be an unramified
character of $\boldsymbol{T.}$ Then the conjugacy classes of hyperspecial
subgroups of $\boldsymbol{G}(k)$ is a principal homogenous space
for a certain finite abelian group $\hat{\Omega}$. Also, the $L$-packet
$\Pi(\varphi_{\lambda})$ associated to $\lambda$ is parametrized
by an abelian group $\hat{R}$. We show that $\hat{R}$ is naturally
a homogenous space for $\hat{\Omega}$. Further, let $\pi_{\rho}\in\Pi(\varphi_{\lambda})$,
where $\rho\in\hat{R}$ and let $[K]$ denote the conjugacy class
of hyperspecial subgroup $K.$ Then we show that $\pi_{\rho}^{K}\neq0$
if and only if $\pi_{\omega\cdot\rho}^{K_{\omega}}\neq0$ where $\omega\in\hat{\Omega}$
and $K_{\omega}$ is any hyperspecial subgroup in the conjugacy class
$\omega\cdot[K]$. 
\end{abstract}
\maketitle
\footnote{2010 Mathematics Subject Classification 22E50 (primary), 20G25, 20F55,
20E42, 11R39 (secondary).%
}

\section*{Introduction\label{intro}}

Let $\boldsymbol{G}$ be a connected reductive group defined over
a local field $k$. The local Langlands conjectures predict that the
irreducible admissible representations of $\boldsymbol{G}(k)$ can
be partitioned into finite sets, known as $L$-packets, in a certain
natural way. Each packet is expected to be associated to what is known
as a Langlands parameter for $\boldsymbol{G}$. 

Now assume $\boldsymbol{G}$ is \textit{unramified}, i.e., it admits
\textit{hyperspecial subgroups}. A representation of $\boldsymbol{G}$
is called unramified if it has a non-zero vector fixed under some
hyperspecial subgroup of $\boldsymbol{G}(k)$. Unramified representations
are of central importance in the Langlands program, as almost all
local components of global representations are unramified. They were
the first representations to be grouped into packets and associated
with Langlands parameters. In this paper, we answer the question of
how $L$-indistinguishability is related to the different choices
of hyperspecial subgroups. 

Let $\boldsymbol{T}$ be a maximal torus in $\boldsymbol{G}$ contained
in a Borel subgroup $\boldsymbol{B}$, both defined over $k$. A character
of $\boldsymbol{T}(k)$ is called unramified if it is trivial on the
maximal compact subgroup $\boldsymbol{T}(k)_{\circ}$ of $\boldsymbol{T}(k)$.
An unramified character $\lambda$ of $\boldsymbol{T}(k)$ corresponds
to a Langlands parameter $\varphi_{\lambda}$ of $\boldsymbol{G}$
via the local Langlands correspondence for tori. The local Langlands
conjectures stipulate what the $L$-packet $\boldsymbol{\Pi}(\varphi_{\lambda})$
associated to such a Langlands parameter $\varphi_{\lambda}$ has
to be. Namely, (see \cite[10.4]{Borel1979}) $\boldsymbol{\Pi}(\varphi_{\lambda})$
should consist precisely of those subquotients of the induced representation
Ind$_{\boldsymbol{B}(k)}^{\boldsymbol{G}(k)}\lambda$, that are unramifi{}ed.
For a hyperspecial subgroup $K$ of $\boldsymbol{G}(k)$, denote by
$\tau_{K,\lambda}$ the unique subquotient of Ind$_{\boldsymbol{B}(k)}^{\boldsymbol{G}(k)}\lambda$
having a $K$-fixed vector. It is well known that the choice of $K$
determines a bijection between the group $\hat{R}_{\varphi_{\lambda}}$
of characters of a certain finite abelian group $R_{\varphi_{\lambda}}$,
and the elements of $\boldsymbol{\Pi}(\varphi_{\lambda})$, with the
trivial character of $R_{\varphi_{\lambda}}$ corresponding to $\tau_{K,\lambda}$.
However, this does not tell us which character of $R_{\varphi_{\lambda}}$
corresponds to which representation in $\boldsymbol{\Pi}(\varphi_{\lambda})$.
This hitherto unexplored question on the fi{}ner internal structure
of the $L$-packet is answered by Theorem \ref{main theorem} (see
Section \ref{main sec}) which, given a character $\rho$ of $R_{\varphi_{\lambda}}$,
specifi{}es the various hyperspecial subgroups $K^{\prime}$ for which
$\tau_{K^{\prime},\lambda}$ corresponds to $\rho$. 

Further, it has long been known that the various $\tau_{K,\lambda}$
as above are all the unramified representations of $\boldsymbol{G}$.
This gets sharpened into a parametrization of the unramified representations
of $\boldsymbol{G}$ when combined with the Corollary \ref{V isom}
to Theorem \ref{main theorem} (see Section \ref{main sec}), which
spells out what the relation between pairs $(K,\lambda)$ and $(K',\lambda^{\prime})$
as above has to be, for $\tau_{K,\lambda}$ to be isomorphic to $\tau_{K^{\prime},\lambda^{\prime}}$.

To put this result in perspective, let us briefl{}y review some history
of previous work related to this problem. When $\boldsymbol{G}$ is
(in addition to being unramified) simply connected and almost simple
and when $\lambda$ is unitary and unramified, D. Keys showed that
all irreducible constituents of the principal series representation
Ind$_{\boldsymbol{B}(k)}^{\boldsymbol{G}(k)}\lambda$ are unramified
\cite[sec 4]{Keys1982}. Without these extra conditions on $\boldsymbol{G}$,
D. Keys and F. Shahidi specified a non-zero Whittaker functional afforded
by $\tau_{K,\lambda}$ \cite[Theorem 4.1]{MR944102}.

A key idea required in the proof of Theorem \ref{main theorem} is
that the Knapp-Stein $R$-group can be realized as the stabilizer,
in a certain subgroup of the affine Weyl group, of a point $x_{\lambda}$
in the Bruhat-Tits building of $G.$ This idea develops over the work
of M. Reeder \cite{MR2674853}. Note that we are allowing $\lambda$
to be non-unitary, so the principal series representation Ind$_{\boldsymbol{B}(k)}^{\boldsymbol{G}(k)}\lambda$
does not, in general decompose as a direct sum of irreducible representations.
The construction of the $L$-packet as described in \cite{Shahidi2011}
is required.

\section{Notations}

Let $\boldsymbol{\boldsymbol{G}}$ be an unramified, connected reductive
linear algebraic group defined over a $p$-adic local field $k$.
We denote the $k$-point of $\boldsymbol{G}$ by $G$ and likewise
for all its subgroups. Let $\boldsymbol{A}$ denote a maximal split
$k$-torus in $\boldsymbol{G}$ and let $\boldsymbol{T}=Z_{\boldsymbol{G}}\boldsymbol{A}$,
the centralizer of $\boldsymbol{A}$ in $\boldsymbol{T}$. Then $\boldsymbol{T}$
is a maximal torus in $\boldsymbol{G}$, since $\boldsymbol{G}$ is
quasi-split. Let $\boldsymbol{B}$ be a Borel subgroup containing
$\boldsymbol{T}.$ Let $\Psi=\left(X,\Phi,\Delta,\check{X},\check{\Phi},\check{\Delta}\right)$
be the based root datum of $(\boldsymbol{G,}\boldsymbol{B,}\boldsymbol{T})$.
Let $W_{k}$ denote the Weil group of $k$. Let $I=I_{k}$ be the
inertia subgroup of $W_{k}$ and let $\sigma=\sigma_{k}$ be the Frobenius
element in $W_{k}/I$. The Frobenius element $\sigma$ induces an
automorphism of $\Psi$ which we again denote by $\sigma$. If $K$
is a compact subgroup of $G$, we denote the set of conjugates of
$K$ in $G$ by $[K].$ If $\pi$ is a representation of $G$, by
the expression $\pi^{[K]}\neq0$, we mean that there is a non-zero
$K$-fixed vector in the space realizing $\pi$ (consequently, for
any conjugate $K^{\prime}\in[K],$ there is a non-zero $K^{\prime}$-fixed
vector). 

A character of $T$ is called unramified, if it is trivial on the
maximal compact subgroup $T_{\circ}$ of $T$. Let $\lambda$ be such
a character of $T.$ To this character, one can associate a Langlands
parameter $\varphi=\varphi_{\lambda}.$ Denote by $\boldsymbol{\Pi}(\varphi_{\lambda})$
the $L$-packet assicated to $\varphi_{\lambda}$. Let $\hat{\boldsymbol{G}}$
be the complex dual of $\boldsymbol{G}$. Let $S_{\varphi}=Z_{\hat{\boldsymbol{G}}}\mbox{Im}(\varphi),$
the centralizer in $\hat{\boldsymbol{G}}$ of the image of $\varphi$
and let $S_{\varphi}^{\circ}$ be its connected component containing
the identity. Let $\hat{Z}=Z(\hat{\boldsymbol{G}})$, the center of
the dual group. Let $R_{\varphi}=R_{\varphi}(\boldsymbol{G)}:=S_{\varphi}/S_{\varphi}^{\circ}\hat{Z}^{\sigma}$
be the $R$$ $ group defined by Langlands. This group turns out to
be abelian. The elments of the $L$-packet $\boldsymbol{\Pi}(\varphi)$
are parametrized by $\hat{R}_{\varphi}$, the group of characters
of $R_{\varphi}$. This parametrization is not canonical, however.
It depends on the choice of a nilpotent orbit. We fix a parametrization
and denote the representation corresponding to a character $\rho\in\hat{R}_{\varphi}$
under this parametrization by $\pi_{\rho}$.

\section{\label{main statement}Statement of the theorems\label{main sec}}

Let $\boldsymbol{G}_{ad}$ be the adjoint group of $\boldsymbol{G}$
and let $\boldsymbol{T}_{ad}$ be the maximal torus of $\boldsymbol{G}_{ad}$
which is the image of $\boldsymbol{T}$ in $\boldsymbol{G}_{ad}$.
The set of conjugacy classes of hyperspecial subgroups of $G$ is
a principal homogenous space for the finite abelian group $\hat{\Omega}:=\mbox{coker}(\check{X}^{\sigma}\rightarrow\check{X}_{ad}^{\sigma})$,
where $\check{X}_{ad}:=\check{X}(\boldsymbol{T}_{ad}).$ The action
is defined as follows: the natural map $\check{X}(\boldsymbol{T}_{ad})^{\sigma}\otimes\mathbb{R}\rightarrow\check{X}(\boldsymbol{A})\otimes\mathbb{R}$
gives an action of $\check{X}_{ad}^{\sigma}$ on the apartment $\mathcal{A}(\boldsymbol{A}).$
Let $\mathcal{H}=\{x\in\mathcal{A}(\boldsymbol{A}):x\mbox{ is hyperspecial}\}.$
Sine $\boldsymbol{G}$ is unramified, we may and do assume that the
hyperspecial subgroups of $G$ are the stabilizers $G_{x}$ of hyperspecial
points $x.$ For $t\in\check{X}_{ad}^{\sigma}$ and $x\in\mathcal{H},$
define $t\cdot[G_{x}]:=[G_{t\cdot x}].$ Then this action is well
defined and it makes the conjugacy classes of hyperspecial subgroups,
a principal homogenous space for $\hat{\Omega}.$
\begin{thm}
\label{main theorem}There exits a canonical surjection $\zeta\colon$$\hat{\Omega}\twoheadrightarrow\hat{R}_{\varphi}$
and 
\[
\pi_{\omega\cdot\rho}^{\omega\cdot[K]}\neq0\iff\pi_{\rho}^{[K]}\neq0,
\]
 for all $\omega\in\hat{\Omega},$ $\rho\in\hat{R}_{\varphi}$, where
$K$ is a hyperspecial subgroup of $G.$ Here $\omega\cdot\rho:=\rho+\zeta(\omega).$ 
\end{thm}
For a hyperspecial subgroup $K$, a representation of $G$ is called
\textit{$K$-spherical} if it is smooth, irreducible, admissible and
contains a non-zero vector invariant under $K$. Write $\boldsymbol{B}=\boldsymbol{TU}$,
where $\boldsymbol{U}$ is the unipotent radical of $\boldsymbol{B}$.
The $K$-spherical representations have the following explicit description:

Define the complex valued function $\varLambda_{K,\lambda}\colon G\rightarrow\mathbb{C}$
\[
\varLambda_{K,\lambda}(tuk)=\lambda(t)\delta^{\frac{1}{2}}(t),
\]
 for $t\in T$ , $u\in U$ and $k\in K$. Here $\delta$ is the modulus
function. Put 
\[
\Gamma_{K,\lambda}(g)=\int_{K}\varLambda_{K,\lambda}(kg)dk,
\]
 for $g\in G$. Denote by $V_{[K],\lambda}$ the space of functions
$f$ on $G$ of the form 
\[
f(g)=\sum_{i=1}^{n}c_{i}\Gamma_{K,\lambda}(gg_{i})
\]
 for $c_{1},\ldots,c_{n}$ in $\mathbb{C}$ and $g_{1},\ldots,g_{n}$
in $G$. We let $G$ operate on $V_{[K],\lambda}$ by right translations,
namely 
\[
(\pi_{[K],\lambda}(g)f)(g_{1})=f(g_{1}g)
\]
 for $f$ in $V_{[K],\lambda}$ and $g$,$g_{1}$ in $G$. Then it
is well known that the representation $(\pi_{[K],\lambda},V_{[K],\lambda})$
is $K$-spherical and any $K$-spherical representation is isomorphic
to $(\pi_{[K],\lambda},V_{[K],\lambda})$ for some unramified character
$\lambda$ (see \cite[sec. 4.4]{MR546593}). 

For different choices of $K,$ the following corollary of theorem
\ref{main theorem} describes the condition for the representations
$(\pi_{[K],\lambda},V_{[K],\lambda})$ to be isomorphic. 
\begin{cor}
\label{V isom}$(\pi_{\omega_{1}\cdot[K],\lambda},V_{\omega_{1}\cdot[K],\lambda})\cong(\pi_{\omega_{2}\cdot[K],\lambda},V_{\omega_{2}\cdot[K],\lambda})\mbox{ if and only if }\zeta(\omega_{1})=\zeta(w_{2}).$
\end{cor}
Let $\tilde{\rho}$ denote any element of $\zeta^{-1}(\rho)$ for
$\rho\in\hat{R}_{\varphi}$. Theorem \ref{main theorem} and corollary
\ref{V isom} immediately imply the following explicit description
of the elements of the $L$-packet $\boldsymbol{\Pi}(\varphi_{\lambda})$.
\begin{cor}
The elements of the L-packet $\boldsymbol{\Pi}(\varphi_{\lambda})$
are $\{(\pi_{\tilde{\rho}\cdot[K],\lambda},V_{\tilde{\rho}\cdot[K],\lambda}):\rho\in\hat{R}_{\varphi}\}.$
\end{cor}

\section{The group $\Omega$\label{sec: group omega}}

Let $\Psi=(X,\Phi,\Delta,\check{X,}\check{\Phi},\check{\Delta})$
be a based root datum as defined in \cite[1.9]{Springer1979}. Let
$W=W(\Psi)$ be the Weyl group. Let $\{\alpha_{1},\ldots,\alpha_{l}\}$
be the set of simple roots $\Delta$. Let $\Phi_{\mbox{aff}}=\{a+n:a\in\Phi$,
$n$ is an integer if $\frac{1}{2}a\notin\Phi$ (resp. any odd integer
if $\frac{1}{2}a\in\Phi)\}$. Then $\Phi_{\mbox{aff}}$ is an affine
root system in the real vector space $V=V(\Psi):=\mathbb{R}\bigotimes\check{X}$
(see \cite[Theorem 1.2.1]{MR1976581}). The Weyl group $W$ acts on
$V$ as the group of reflections genrerated by $r_{1},\ldots,r_{l}$,
where $r_{i}$ fixes the hyperplace $\{x\in V:\alpha_{i}(x)=0)\}$.
The basis $\Delta$ determines a particular \textit{alcove} $C$ in
$V$ as follows: let $\tilde{\alpha_{0}}=\sum_{i=1}^{l}a_{i}\alpha_{i}$
be the highest root with respect to $\Delta.$ Let $\alpha_{0}$ be
the affine linear function $1-\tilde{\alpha}_{0}$ on $V$ and set
$a_{0}=1,$ so that 
\[
\sum_{i=0}^{l}a_{i}\alpha_{i}\equiv1.
\]
Then the alcove determined by $\Delta$ is the intersection of the
half spaces:
\[
C=\{x\in V:\alpha_{i}(x)>0\mbox{ for }0\leq i\leq l\}.
\]

The group $\tilde{W}=W\ltimes\check{X}$ is called the \textit{extended
affine Weyl group}. It contains the \textit{affine Weyl group} $\tilde{W}^{\circ}=W\ltimes\mathbb{Z}\check{\Phi}$
as a normal subgroup. $\tilde{W}$ acts transitively on the set of
alcoves in $V.$ Let 
\[
\Omega=\Omega(\Psi):=\left\{ \rho\in\tilde{W}:\rho\cdot C=C\right\} .
\]
Then, 
\[
\tilde{W}=\Omega\ltimes\tilde{W}^{\circ}
\]
and therefore $\Omega\cong\tilde{W}/\tilde{W}^{\circ}\cong\check{X}/\mathbb{Z}\check{\Phi}$.

\section{\label{R group theory}R-groups and intertwinig operators}

Let $\boldsymbol{U}$ denote the unipotent redical of $\boldsymbol{B}$.
Then $\boldsymbol{B}=\boldsymbol{T}\boldsymbol{U}.$ We will abbreviate
by $I(\lambda)$, the principal series representation $\mbox{Ind}_{B}^{G}\lambda$
of $\boldsymbol{G}.$ Let $\boldsymbol{N}:=N_{\boldsymbol{G}}\boldsymbol{A}$
and let $W$ be the relative Weyl group of $\boldsymbol{G}$. Then
$W\cong\boldsymbol{N}/\boldsymbol{T}$. Denote by $\bar{w}$, a representative
of an element $w$ of $W.$ Let $\bar{\boldsymbol{B}}=\boldsymbol{T}\bar{\boldsymbol{U}}$
be the Borel subgroup opposite to $\boldsymbol{B}$. Define the usual
intertwining operators :
\[
\mathrm{\boldsymbol{A}}(w,\lambda)\colon I(\lambda)\rightarrow I(w\lambda),
\]
\[
(\mathrm{\boldsymbol{A}}(\bar{w},\lambda)f)(g)=\int_{\boldsymbol{U}\cap w\bar{\boldsymbol{U}}w^{-1}}f(gu\bar{w})du.
\]
 These operators converge in an appropriate domain and they may be
analytically continued so that they are defined for all unitary $\lambda$
\cite[Sec. 2]{keys1982reducibility}. Further, they satisfy the co-cycle
relation
\begin{equation}
\mathrm{\boldsymbol{A}}(\bar{w}_{1}\bar{w}_{2},\lambda)=\mathrm{\boldsymbol{A}}(\bar{w}_{1},\bar{w}_{2}\lambda)\mathrm{\boldsymbol{A}}(\bar{w}_{2},\lambda),
\end{equation}
 provided 
\begin{equation}
\mbox{length}(w_{1}w_{2})=\mbox{length}(w_{1})+\mbox{length}(w_{2}).
\end{equation}
 Define the normalized intertwining operators 
\begin{equation}
\mathscr{A}(\bar{w},\lambda):=\frac{1}{c_{w}(\lambda)}\mathrm{\boldsymbol{A}}(\bar{w},\lambda),
\end{equation}
 where $c_{w}(\lambda)$ is the Harish-Chandra $c$-function. Then
with appropriate choices of the coset representatives $\bar{w}$ of
the Weyl group elements, the co-cycle relation 
\begin{equation}
\mathrm{\mathscr{A}}(\bar{w}_{1}\bar{w}_{2},\lambda)=\mathrm{\mathrm{\mathscr{A}}}(\bar{w}_{1},\bar{w}_{2}\lambda)\mathrm{\mathscr{A}}(\bar{w}_{2},\lambda),
\end{equation}
 holds without any condition on the lengths of $w_{1}$ and $w_{2}$. 

Let $\Phi_{rel}$ be the relative roots. Then $W$ is the Weyl group
associated to $\Phi_{rel}$. For $\alpha\in\Phi_{rel}$, let $s_{\alpha}$
denote the reflection associated to $\alpha$. Define
\[
\Delta^{\prime}=\{\alpha\in\Phi_{rel}:A(w,\lambda)\mbox{ is scalar}\},
\]
 and let $W^{\prime}$ be the reflection group $<s_{\alpha}:\alpha\in\Delta^{\prime}>$.
Let
\[
W_{\lambda}=\{w\in W:w\lambda=\lambda\}.
\]
 Define 
\begin{equation}
R=R(\boldsymbol{G},\lambda)=\{w\in W_{\lambda}:\alpha\in\Delta^{\prime}\mbox{ and }\alpha>0\mbox{ imply }w\alpha>0\}.
\end{equation}
 Then 
\begin{equation}
W_{\lambda}=W^{\prime}\rtimes R.
\end{equation}
 For a unitary, unramified character, the Knapp-Stein $R$-group is
abelian. This is shown in \cite{keys1982reducibility} for simply
connected, almost simple, semi-simple groups. In Corollary \ref{R-group corr},
we have proved it for any connected reductive group. It then follows
from \cite[Thm. 2.4]{Keys1987} ,
\begin{thm}
For an unramified unitary character $\lambda$ of $T$,\end{thm}
\begin{enumerate}
\item The commuting algebra End$(I(\lambda))$ of $I(\lambda)$ is isomorphic
to the group algebra $\mathbb{C}[R]$.
\item $I(\lambda)$ decomposes with multiplicity one.
\item The irreducible components of $I(\lambda)$ are parametrized by the
characters of $R$. 
\end{enumerate}
Let $I\left(\lambda\right)\cong\underset{\rho\in\hat{R}}{\oplus}\pi_{\rho}$
be a decomposition of the pricipal series, where $\hat{R}$ is the
group of characters of $R.$ This parametrization is not unique. It
depends on the choice of normalizations of intertwining operators
giving End($I(\lambda))\cong\mathbb{C}[R].$ But any two normalizations
must differ by a one dimentional character $\rho^{\prime}$ of $R$,
i.e.,
\[
\mathscr{A}^{\prime}(r,\lambda)=\rho^{\prime}(r)\mathscr{A}(r,\lambda).
\]
Then if $\pi$ corresponds to the character $\rho$ in the first normalization,
then it corresponds to the character $\rho+\rho^{\prime}$ in the
second normalization.

\section{\label{image under LLC}Image of an unramified character under LLC}

In this section, we review the description of the image of an unramified
character, under the local Langlands correspondence for tori.

Let $\boldsymbol{T}$ be a torus defined over $k$. Assume that $\boldsymbol{T}$
splits over an unramifed extension $k^{\prime}$ of $k$. As before,
we denote by $T_{\circ}$, the maximal compact subgroup of $T$. Given
$t\in\boldsymbol{T}(k^{\prime})$, let $\nu(t)\in\mbox{Hom}(X^{*}(\boldsymbol{T}),\mathbb{Z})$
be defined by $\nu(t)(m)=\mbox{ord}(m(t))$. Then the map $t\mapsto\nu(t)$
induces an isomorphism (see \cite[Sec 9.5]{Borel1979}):
\begin{eqnarray*}
T/T_{\circ} & \cong & X_{*}(\boldsymbol{T})^{\sigma}\\
 & = & X^{*}(\hat{\boldsymbol{T}})^{\sigma}\\
 & \cong & X^{*}(\hat{\boldsymbol{T}}_{\sigma}).
\end{eqnarray*}
Here $\hat{\boldsymbol{T}}_{\sigma}$ denotes the co-invariant of
$\hat{\boldsymbol{T}}$ with respect to $\sigma$. From this we get
\begin{eqnarray}
\mbox{Hom}(T/T_{\circ},\mathbb{C}^{\times}) & \cong & \mbox{Hom}(X^{*}(\hat{\boldsymbol{T}}_{\sigma}),\mathbb{C}^{\times})\nonumber \\
 & \cong & \hat{\boldsymbol{T}}_{\sigma}.\label{eq:satake-1}
\end{eqnarray}
Thus to an unramified character $\chi$ of $T,$ we can associate
the $\sigma$-conjugacy class of a semisimple element $s_{\chi}$
of $\hat{\boldsymbol{T}}$. In fact, this is the image of $\chi$
under the Local Langlands correspondence for tori, 
\[
\xymatrix{\mbox{Hom}(T,\mathbb{C}^{\times})\ar@{->}[r]_{\sim}^{LLC} & H^{1}(W_{k},\hat{\boldsymbol{T})}\\
\mbox{Hom}(T/T_{\circ},\mathbb{C}^{\times})\ar@{^{(}->}[u]\ar@{->}[r]^{\sim} & \hat{\boldsymbol{T}}_{\sigma}\cong H^{1}(W_{k}/I,\hat{\boldsymbol{T}})\ar@{^{(}->}[u]^{infl}
}
\]

\section{Coefficient of the spherical vector}

Assume $\lambda$ to be unitary. Let $K$ be a hyperspecial subgroup
of $G$. Let $f_{K}\in I(\lambda)^{K}$ be such that $f_{K}(1)=1.$
Let $\rho_{\lambda,[K]}$ be the character of $R$ defined by $\mathscr{A}\left(r,\lambda\right)f_{K}=\rho_{\lambda,[K]}\left(r\right)f_{K}$.
Then,
\begin{lem}
\label{lem pi K-sph }$\pi_{\rho}^{[K]}\neq0\iff\rho_{\lambda,[K]}=\rho$.\end{lem}
\begin{proof}
The operators 

\begin{equation}
\mathcal{P}_{\rho}=\frac{1}{|R|}\sum_{r\in R}\rho\left(r\right)^{-1}\mathscr{A}\left(r,\lambda\right)
\end{equation}

are $\hat{|R|}$ orthogonal projections onto the invariant subspaces
of $I(\lambda).$ Let $U_{\rho}=\mathcal{P}_{\rho}U$, where $U$
is the space realizing $I(\lambda)$. Then

\[
U=\bigoplus_{\rho\in\hat{R}}U_{\rho}.
\]

Since

\[
\mathscr{A}\left(r,\lambda\right)f_{K}=\rho_{\lambda,[K]}\left(r\right)f_{K}.
\]
 Therefore

\begin{eqnarray*}
\mathcal{P}_{\rho}f_{K} & = & \frac{1}{|R|}\left(\sum_{r\in R}\rho\left(r\right)^{-1}\rho_{\lambda,[K]}\left(r\right)\right)f_{K}\\
 & = & \begin{cases}
0 & \rho_{\lambda,[K]}\neq\rho\\
f_{K} & \rho_{\lambda,[K]}=\rho,
\end{cases}
\end{eqnarray*}

$i.e.,$ $U_{\rho}^{K}\neq0$ $\iff$$\rho_{\lambda,[K]}=\rho$. 
\end{proof}
Thus, to understand which representation in the $L$-packet $\Pi(\varphi_{\lambda})$
is spherical for which hyperspecial conjugacy class, we need to understand
how the charactes $\rho_{\lambda,[K]}$ relate to each other for different
choices of conjugacy classes $[K]$.

\section{\label{split case}Proof in a special case}

To give the main ideas of the proof quickly, we first prove our theorem
for the simplest case when $\boldsymbol{G}$ is semisimple, almost
simple, simply connected and split and when $\lambda$ is unitary.
We will use the notations introduced in section \ref{sec: group omega}.

The local Langlands correspondence for tori induces an isomorphism,
\[
\mbox{Hom}(T/T_{\circ},\mathbb{C}^{\times})\cong\hat{\boldsymbol{T}}.
\]
 The image of the unramified unitary characters under this isomorphism
is the subtorus $\hat{\boldsymbol{U}}\cong X\otimes\mathbb{S}^{1}\subset\hat{\boldsymbol{T}}\cong X\otimes\mathbb{C}^{\times}$.
Now $\hat{\boldsymbol{U}}\cong X\otimes(\mathbb{R}/\mathbb{Z})\cong X\otimes\mathbb{R}/X$.
Therefore $\hat{\boldsymbol{U}}/W\cong X\otimes\mathbb{R}/W\ltimes X=V(\check{\Psi})/\tilde{W}.$
Thus, the LLC induces a natural ismorphism $\mbox{Hom}(T/T_{\circ},\mathbb{S}^{1})/W\cong V(\check{\Psi})/\tilde{W}$.
Here $\mathbb{S}^{1}$ denotes the unit circle in $\mathbb{C}$. Thus,
to a unitary unramified character $\lambda$, we can naturally associate
a semisimple element $s=s_{\lambda}\in\hat{\boldsymbol{T}}$ which
corresponds to a point $x=x_{\lambda}$ in the closure of the alcove
$C(\check{\Psi})$. We'll denote $\Omega(\check{\Psi})$ by $\Omega$
whenever there is no ambiguity. Let $\Omega_{x}$ denote the stabilizer
of $x$ in $\Omega$. 
\begin{prop}
There is a natural isomorphism $\Omega_{x}\cong R.$\end{prop}
\begin{proof}
By \cite[Prop 2.1]{MR2674853}, $\Omega_{x}\cong\pi_{0}\left(Z_{\boldsymbol{\hat{G}}}s\right)$,
the connected component group of the centralizer of the semi-simple
element $s$. By \cite[Proposition 2.6]{Keys1987}, $\pi_{0}\left(Z_{\boldsymbol{\hat{G}}}s\right)\cong R.$
Thus, $\Omega_{x}\cong R$. Infact, it follows from the proofs of
\cite[Prop 2.1]{MR2674853} \& \cite[Prop. 2.6]{Keys1987} that isomorphism
is induced by the natural projection $\tilde{W}\twoheadrightarrow W$.
We defer these details until proposition \ref{prop- Omega relates R},
where we prove our statement for more general groups.
\end{proof}
Let $\check{P}$ be the co-weight lattice of $\boldsymbol{G}$. Then
the set of conjugacy classes of hyperspecial points of $G$ form a
principal homogenous space for the group $\check{P}/\check{X}$. This
group is in duality with the group $\Omega\cong X/\mathbb{Z}\Phi$.
Denote by $\left(,\right):(\check{P}/\check{X})\times\Omega\rightarrow\mathbb{Q}/\mathbb{Z}$
the pairing between them. Using the last proposition, we can realize
$R$ as a subgroup of $\Omega$ by the natural embedding $R\cong\Omega_{x}\subset\Omega$.
For $r\in R,$ and $\omega\in\check{P}/\check{X}$, let $\rho_{\omega}\left(r\right)=e^{-2\pi i\left(\omega,r\right)}$.
Then we have a natural surjection $\hat{\Omega}\twoheadrightarrow\hat{R}$
given by $\omega\mapsto\rho_{\omega}$. Let $K_{0}$ be a hyperspecial
subgroup satisfying $\pi_{0}^{[K_{0}]}\neq0.$ Pick $K_{\omega}\in\omega\cdot[K_{0}]$
and let $f_{\omega}\in I(\lambda)^{K_{\omega}}$ be the spherical
vector such that $f_{\omega}\left(1\right)=1$.
\begin{prop}
\label{intertwining coeff} $\mathscr{A}(r,\lambda)f_{\omega}=\rho_{\omega}(r)f_{\omega}.$\end{prop}
\begin{proof}
By \cite[Section 4 lemma]{keys1982reducibility}, one can do a case
by case computation to show this. The calculations are shown in the
section \ref{main prop proof}.
\end{proof}
From the surjection $\hat{\Omega}\twoheadrightarrow\hat{R}$, we get
a natural action of $\hat{\Omega}$ on $\hat{R}$, namely, for $\rho\in\hat{R}$
and $\omega\in\hat{\Omega}$, we let $\omega\cdot\rho=\rho+\rho_{\omega}$.
Let $K$ be any hyperspecial subgroup. Let $\omega^{\prime}\in\hat{\Omega}$
be such that $\omega^{\prime}\cdot[K_{0}]=[K]$. Then from lemma \ref{lem pi K-sph },
it follows:
\begin{eqnarray*}
\pi_{\omega\cdot\rho}^{\omega\cdot[K]}\neq0 & \iff & \omega\cdot\rho=\rho_{\omega^{\prime}+\omega}\\
 & \iff & \rho=\rho_{\omega^{\prime}}\\
 & \iff & \pi_{\rho}^{\omega^{\prime}\cdot[K_{0}]}\neq0\\
 & \iff & \pi_{\rho}^{[K]}\neq0.
\end{eqnarray*}

\section{\label{Relative root system}Construction of the relative based root
datum}

Let $\Psi=\left(X,\Phi,\Delta,\check{X},\check{\Phi},\check{\Delta}\right)$
be the based root datum of $(\boldsymbol{G,}\boldsymbol{B,}\boldsymbol{T})$,
where $\boldsymbol{G}$ is an unramifed, connected reductive group,
$\boldsymbol{B}$ is a borel subgroup of $\boldsymbol{G}$ containing
a maximal torus $\boldsymbol{T}$. The follwing construction is given
in \cite{Yu}. 

Define a new 6-touple: 

\begin{eqnarray}
Y & = & (X_{\sigma})/torsion,\nonumber \\
\check{Y} & = & (\check{X})^{\sigma},\nonumber \\
S & = & \{\underbar{\ensuremath{a}}:a\in\Phi\},\mbox{ where }\underbar{\ensuremath{a}}=a|_{\check{Y}},\nonumber \\
\check{S} & = & \{\check{\alpha}:\alpha\in S\},\nonumber \\
E & = & \{\underbar{\ensuremath{a}}:a\in\Delta\},\nonumber \\
\check{E} & = & \{\check{\alpha}\colon\alpha\in E\}.\label{eq:6-touple}
\end{eqnarray}

The explanation for the defining formulas is as follows. We first
note that $Y$ and $\check{Y}$ are free abelian groups, dual to each
other under the canonical pairing $(\underbar{\ensuremath{x}},v)\mapsto<x,v>$,
for $\underbar{\ensuremath{x}}\in Y$, $v\in\check{Y}$, where $x$
is any preimage of $\underbar{\ensuremath{x}}$ in $X$. Define $\check{\alpha}$
for $\alpha\in S$ as follows:
\begin{equation}
\check{\alpha}=\begin{cases}
\underset{a\in\Phi:a|_{\check{Y}}=\alpha}{\sum\check{a},} & \mbox{if }2\alpha\notin S\\
\underset{a\in\Phi:a|_{\check{Y}}=\alpha}{2\sum\check{a},} & \mbox{if }2\alpha\in S
\end{cases}
\end{equation}
 
\begin{thm}
\cite{Yu}The 6-touple $\underbar{\ensuremath{\Psi}}=(Y,S,E,\check{Y},\check{S},\check{E})$,
with the canonical pairing between $Y$ and $\check{Y}$ and the correspondence
$S\rightarrow\check{S}$, $\alpha\mapsto\check{\alpha}$, is a based
root datum, which is not necessarily reduced. Moreover, \end{thm}
\begin{enumerate}
\item The homomorphism $W(\Psi)^{\sigma}\rightarrow\mathrm{\boldsymbol{GL}}(\check{Y})$,
$w\mapsto w|_{\check{Y}}$ is injective and the image is $W(\underbar{\ensuremath{\Psi}})$.
\item Let $a,b\in\Phi$. Then $\underbar{\ensuremath{a}}=\underbar{\ensuremath{b}}$
if and only if $a,b$ are in the same $\sigma$-orbit. 
\end{enumerate}

\section{Geometric description of the $R$-group}

Let $\underbar{\ensuremath{\check{\Psi}}}$ denote the dual of the
based root datum $\underbar{\ensuremath{\Psi}}$ defined in section
\ref{Relative root system}. Let $\underbar{V}=V(\check{\underbar{\ensuremath{\Psi}}})=Y\otimes\mathbb{R}$
and let $\tilde{W}=W\ltimes Y$, where $W=W(\underbar{\ensuremath{\Psi}})$
is the relative Weyl group. By equation \ref{eq:satake-1}, we have
$\mbox{Hom}(T/T_{\circ},\mathbb{C}^{\times})\cong\hat{\boldsymbol{T}}_{\sigma}$.
Under this isomorphism, the unramified unitary characters will correspond
to the $\sigma$-conjugacy classes of the compact subtorus $\hat{\boldsymbol{U}}\cong X\otimes\mathbb{S}^{1}$
of $\hat{\boldsymbol{T}}\cong X\otimes\mathbb{C}^{\times}$. We have,
\begin{eqnarray*}
\hat{\boldsymbol{U}} & \cong & X\otimes(\mathbb{R}/\mathbb{Z})\\
 & \cong & X\otimes\mathbb{R}/X.
\end{eqnarray*}

Therefore,
\begin{eqnarray*}
\boldsymbol{\hat{U}}_{\sigma}/W & \cong & Y\otimes\mathbb{R}/W\ltimes Y\\
 & \cong & \underbar{V}/\tilde{W}.
\end{eqnarray*}

By equation \ref{eq:satake-1}, it follows that the class of unitary
unramified characters in $\mbox{Hom}(T/T_{\circ},\mathbb{C}^{\times})/W$
are in one to one correspondence with the points of $\underbar{V}/\tilde{W}$.
Assume that the unramified character $\lambda$ of $T$ is unitary.
To the $W$ orbit of $\lambda$ we can therefore associate a point
$x=x_{\lambda}$ in the closure of the alcove $\underbar{C}:=C(\check{\underbar{\ensuremath{\Psi}}})$
of $\underbar{V}$. 

Let $\underbar{\ensuremath{\Omega}}:=\Omega(\check{\underbar{\ensuremath{\Psi}}})$.
We will also denote $\Omega(\check{\underbar{\ensuremath{\Psi}}})$
by $\Omega_{G}$ when the based root-datum is clear from the context.

We have $\underbar{\ensuremath{\Omega}}\cong Y/\mathbb{Z}E.$ Let
$\underbar{\ensuremath{\Omega}}_{x}$ be the stabilizer of $x$ in
$\underbar{\ensuremath{\Omega}}$. 
\begin{prop}
\label{prop- Omega relates R}The natural projection $\tilde{W}\twoheadrightarrow W$
induces an isomorphism\linebreak{}
 $\underbar{\ensuremath{\Omega}}_{x}\cong R$. \end{prop}
\begin{proof}
Let $\tilde{W}_{x}=\{w\in\tilde{W}:w\cdot x=x\}$ be the stabilizer
of $x$ in $\tilde{W}$. This group is finite and its normal subgroup
$\tilde{W}_{x}^{\circ}$, generated by reflections about the hyperplanes
through $x$, acts simply transitively on the set of alcoves containing
$x$ in their closure. It follows that
\[
\tilde{W}_{x}=\Omega_{x}\ltimes\tilde{W}_{x}^{\circ}
\]
 and therefore, 
\begin{equation}
\underbar{\ensuremath{\Omega}}_{x}\cong\tilde{W}_{x}/\tilde{W}_{x}^{\circ}.\label{eq:-2}
\end{equation}
 Let $\pi\colon\tilde{W}\rightarrow W$ be the projection map. Let
$\delta=s\rtimes\sigma$. Let $\hat{\boldsymbol{N}}=N_{\hat{\boldsymbol{G}}}\hat{\boldsymbol{T}}$
be the normalizer of $\hat{\boldsymbol{T}}$ in $\hat{\boldsymbol{G}.}$
Let $S_{\delta}=Z_{\hat{\boldsymbol{G}}}\delta$. Let $\hat{\boldsymbol{Z}}=Z(\hat{\boldsymbol{G)}}$
be the center of $\hat{\boldsymbol{G.}}$ Then the proof in \cite[lemma 3.9]{MR2674853}
shows :
\begin{equation}
\tilde{W}_{x}\cong\hat{\boldsymbol{N}}\cap S_{\delta}/\hat{\boldsymbol{T}}^{\sigma},\label{W_x 1}
\end{equation}
\begin{equation}
\tilde{W}_{x}^{\circ}\cong\hat{\boldsymbol{N}}\cap S_{\delta}^{\circ}/(\hat{\boldsymbol{T}}^{\sigma})^{\circ}.\label{W_x 2}
\end{equation}
We will include the proofs of \ref{W_x 1} and \ref{W_x 2} here for
the sake of completeness. Let $W_{\delta}:=\hat{\boldsymbol{N}}\cap S_{\delta}/\hat{\boldsymbol{T}^{\sigma}}=\hat{\boldsymbol{N}}^{\delta}/\hat{\boldsymbol{T}^{\sigma}}$
be the subgroup of $W$ whose elements can be represented by $\delta$
fixed elements of $\hat{\boldsymbol{N}}$. If $w\in W$ is the projection
of an element of $\tilde{W}_{x}$ then $w\cdot x-x\in Y$. Let $f$
be the order of $\sigma$ and let 
\[
P_{\sigma}=f^{-1}(1+\sigma+\ldots+\sigma^{f-1})\in\mbox{End}(V)
\]
 be the projection onto $\underbar{V}$. Here $V=V(\check{\Psi})$
as before. Then $Y\cong P_{\sigma}X$ and 
\begin{equation}
[(1-\sigma)V+X]\cap\underbar{V}=Y.\label{V,X identity}
\end{equation}
\cite[Lemma 3.4 (17)]{MR2674853}. There exists therefore, $v\in V$
and $y\in X$ such that 
\[
w\cdot x-x=(\sigma-1)v+y.
\]
 Setting $s=\mbox{exp}(x),$ $t=\mbox{exp}(v)$, we have 
\[
w(s)=t^{-1}\sigma(t)s.
\]
 By \cite[Lemma 6.2]{Borel1979}, we may choose $\dot{w}\in\hat{\boldsymbol{N}}^{\sigma}$
such that $w=\dot{w}\hat{\boldsymbol{T}}$. Then the element $n=t\dot{w}$
belongs to $\hat{\boldsymbol{N}}^{\sigma}$ and $n\hat{\boldsymbol{T}}=w$.
Thus the projection maps $\tilde{W}_{x}$ into $W_{\delta}$. The
argument is reversible showing that $\pi(\tilde{W}_{x})=W_{\delta}$.
Finally, since the kernel of $\pi$ is torsion free and $\tilde{W}_{x}$
is finite, the map is also injective, completing the proof of \ref{W_x 1}. 

Let $W_{\delta}^{\circ}=\pi(\tilde{W}_{x}^{\circ})$. Let $A_{f}$
be the set of affine roots vanishing on the hyperplanes passing through
$x$ and let $A_{f}^{v}$ be their vector parts. Then $W_{\delta}^{\circ}$
is the subgroup of $W_{\delta}$ generated by the roots $A_{f}^{v}$
and $W_{\delta}^{\circ}=\hat{\boldsymbol{N}}\cap S_{\delta}^{\circ}/(\hat{\boldsymbol{T}}^{\sigma})^{\circ}$
is the Weyl group $W(S_{\delta}^{\circ},(\hat{\boldsymbol{T}}^{\sigma})^{\circ})$
\cite[3.5]{MR794307}.

Let $R_{\lambda}=S_{\delta}/S_{\delta}^{\circ}\hat{\boldsymbol{Z}}^{\sigma}$.
Then Proposition 2.6 in \cite{Keys1987} states that there is a short
exact sequence
\begin{equation}
1\rightarrow\hat{\boldsymbol{N}}\cap S_{\delta}^{\circ}/(\hat{\boldsymbol{T}}^{\sigma})^{\circ}\rightarrow\hat{\boldsymbol{N}}\cap S_{\delta}/\hat{\boldsymbol{T}}^{\sigma}\rightarrow R_{\lambda}\rightarrow1\label{eq:3}
\end{equation}
 and $R\cong R_{\lambda}$. Therefore, 
\begin{eqnarray*}
R & \cong & (\hat{\boldsymbol{N}}\cap S_{\delta}/\hat{\boldsymbol{T}}^{\sigma})/(\hat{\boldsymbol{N}}\cap S_{\delta}^{\circ}/(\hat{\boldsymbol{T}}^{\sigma})^{\circ})\\
 & \cong & \tilde{W}_{x}/\tilde{W}_{x}^{\circ}\\
 & \cong & \underbar{\ensuremath{\Omega}}_{x}.
\end{eqnarray*}
The last isomorphism follows from equation \ref{eq:-2}. This completes
the proof of the proposition. 
\end{proof}
This result immediately gives the classification of $R$-groups. The
same classification is obtained in \cite[Sec. 3, Theorem]{Keys1982},
in a case by case manner. In particular, we have shown:
\begin{cor}
\label{R-group corr}The Knapp-Stein $R$-group for an unramified
unitary character is abelian. \end{cor}
\begin{lem}
\label{Omega X rel}$\widehat{\underbar{\ensuremath{\Omega}}^{tor}}=\mbox{coker}(\check{X}^{\sigma}\rightarrow\check{X}_{ad}^{\sigma})$. \end{lem}
\begin{proof}
Let $L$ denote the lattice $\mathbb{Z}E$. We have the short exact
sequence 
\[
0\rightarrow L\rightarrow Y\rightarrow Y/L\rightarrow0.
\]
 Applying the contravariant functor $\mbox{Hom}(-,\mathbb{Z})$ to
this, we get 
\[
\check{Y}\rightarrow\check{L}\rightarrow\mbox{Ext}{}^{1}(Y/L,\mathbb{Z})\rightarrow0.
\]
 This implies
\begin{eqnarray}
\mbox{Ext}{}^{1}(Y/L,\mathbb{Z}) & \cong & \check{L}/\mbox{im}(\check{Y})\\
 & \cong & \check{X}_{ad}^{\sigma}/\mbox{im}(\check{X}^{\sigma}).\label{lattice eq-1}
\end{eqnarray}
 Also, we have the short exact sequence 
\[
0\rightarrow\mathbb{Z}\rightarrow\mathbb{Q}\rightarrow\mathbb{Q}/\mathbb{Z}\rightarrow0.
\]
 Let $A$ be an abelian group. Applying the functor $\mbox{Hom}(A,-)$,
we get, 
\[
\mbox{Hom}(A,\mathbb{Q}/\mathbb{Z})/\mbox{im}(\mbox{Hom}(A,\mathbb{Q}))\cong\mbox{Ext}{}^{1}(A,\mathbb{Z}).
\]
 When $A$ is finitely generated, we also have that
\[
0\rightarrow\mbox{im}(\mbox{Hom}(A,\mathbb{Q}))\rightarrow\mbox{Hom}(A,\mathbb{Q}/\mathbb{Z})\rightarrow\mbox{Hom}(A^{tor},\mathbb{Q}/\mathbb{Z})\rightarrow0.
\]
 Here $A^{tor}$ denotes the torsion part of $A$. From the last sequence,
we get that, 
\begin{equation}
\widehat{A^{tor}}\cong\mbox{Ext}{}^{1}(A,\mathbb{Z}).\label{lattice eq-2}
\end{equation}
 Putting $A=Y/L$ and using equations \ref{lattice eq-1} and \ref{lattice eq-2},
we get,
\begin{eqnarray}
\widehat{(Y/L)^{tor}} & \cong & \mbox{Ext}{}^{1}(Y/L,\mathbb{Z})\\
 & \cong & \check{X}_{ad}^{\sigma}/\mbox{im}(\check{X}^{\sigma}).
\end{eqnarray}
 This completes the proof of the lemma.
\end{proof}

\section{\label{G-semisimple-unramified.}Proof of the main theorem}

We assume $\lambda$ to be unitary till the end of section \ref{G reductive unramified}.

\subsection{\label{sub: G sc}G semisimple, simply connected and unramified}

Assume $\boldsymbol{G}$ to be unramified, semisimple and simply connected.
Then we can write $\boldsymbol{G}$ as the direct product
\[
\boldsymbol{G=}\boldsymbol{H}_{1}\times\ldots\times\boldsymbol{H}_{n},
\]
 where $\boldsymbol{H}_{i}$ are semisimple, almost simple, simply
connected and unramified. Thus, it suffices to prove the result when
$\boldsymbol{G}$ is unramified, semisimple, almost simple, simply
connected.

Let $\check{P}$ be the co-weight lattice of $\boldsymbol{G}$. Then
the set of conjugacy classes of hyperspecial points of $G$ form a
principal homogenous space for the group $\check{P}^{\sigma}/\check{X}^{\sigma}$.
By lemma \ref{Omega X rel}, this group is dual to the group $\underbar{\ensuremath{\Omega}}\cong Y/\mathbb{Z}E$
(notations as in equation \ref{eq:6-touple}). Denote by $\left(,\right):(\check{P^{\sigma}}/\check{X}^{\sigma})\times\underbar{\ensuremath{\Omega}}\rightarrow\mathbb{Q}/\mathbb{Z}$
the pairing between them. By proposition \ref{prop- Omega relates R},
$R\cong\underbar{\ensuremath{\Omega}}_{x}$. So we can realize $R$
as a subgroup of $\underbar{\ensuremath{\Omega}}$. Thus we have a
pairing $\left(,\right):(\check{P^{\sigma}}/\check{X}^{\sigma})\times R\rightarrow\mathbb{Q}/\mathbb{Z}$.
For $r\in R,$ and $\omega\in\check{P}^{\sigma}/\check{X}^{\sigma}$,
let $\rho_{\omega}\left(r\right)=e^{-2\pi i\left(\omega,r\right)}$.
Then we have a natural surjection $\hat{\underbar{\ensuremath{\Omega}}}\twoheadrightarrow\hat{R}$
given by $\omega\mapsto\rho_{\omega}$. Denote by $[K_{0}]$ the conjugacy
class of a hyperspecial subgroup $K_{0}$ satisfying $\pi_{0}^{[K_{0}]}\neq0.$
Pick $K_{\omega}\in\omega\cdot[K_{0}]$ and let $f_{\omega}\in I(\lambda)^{K_{\omega}}$
be the spherical vector such that $f_{\omega}\left(1\right)=1$.
\begin{prop}
\label{intertwining coeff-1} $\mathscr{A}(r,\lambda)f_{\omega}=\rho_{\omega}(r)f_{\omega}$.\end{prop}
\begin{proof}
The proof again is computation like in the split case. It is worked
out in section \ref{main prop proof}.
\end{proof}
Then from the lemma \ref{lem pi K-sph }, we get that $\pi_{\rho}^{\omega\cdot[K_{0}]}\neq0$
iff $\rho=\rho_{\omega}$. The rest of the proof of the theorem in
this case is now identical to the split case.

\subsection{\label{G reductive unramified}G reductive and unramified}

Let $\tilde{\boldsymbol{G}}=\boldsymbol{Z}^{\circ}\times(\boldsymbol{G}^{der})^{sc}$
where $\boldsymbol{G}^{der}$ is the derived group of $\boldsymbol{G}$
and $(\boldsymbol{G}^{der})^{sc}$ is its simply connected cover.
There is an isogeny $\varsigma:\boldsymbol{\tilde{G}\twoheadrightarrow G}$
. Let $\boldsymbol{\tilde{T}}=\boldsymbol{Z^{\circ}}\times(\boldsymbol{T}^{der})^{sc}$.
Let $\tilde{\lambda}$ be the pull back of the character $\lambda$
of $T$ to $\tilde{T}$ . If $f\in I(\lambda),$ then $f\circ(\varsigma|_{\tilde{G}})$
is a map $\tilde{G}\rightarrow\mathbb{C}$ and it satisfies the principal
series condition. Thus, $\varsigma$ induces 
\[
\varrho:I(\lambda)\rightarrow I(\tilde{\lambda})
\]
a $\tilde{G}$ map of the principal series representations. 
\begin{lem}
\label{lem: principal surjective.}$\varrho$ is surjective. \end{lem}
\begin{proof}
Let $\bar{f}$ (resp $\bar{g}$) denote the image of an element $f\in I(\tilde{\lambda})$
(resp $g\in\tilde{G})$ under $\varrho$ (resp $\varsigma$). Let
$W$ be a represention of $\tilde{G}$ such that

\[
I(\tilde{\lambda})=\overline{I(\lambda)}\oplus W.
\]
If $W$ is non-trivial, then there exists a hyperspecial subgroup
$\tilde{K}$ of $\tilde{G}$ such that $W^{\tilde{K}}\neq0$ \cite[Sec. 4, Theorem]{keys1982reducibility}.
Let $K$ be a hyperspecial subgroup of $G$ such that $K\supset\bar{\tilde{K}}$,
where $\bar{\tilde{K}}=\varrho(\tilde{K})$. Then $I(\lambda)^{K}\neq0$
implies $\overline{I(\lambda)}^{\tilde{K}}\neq0.$ But dim $I(\tilde{\lambda})^{\tilde{K}}=1$,
a contradiction. This implies $W=0$ and therefore the claim. \end{proof}
\begin{lem}
\label{lemma reduct omega }There is a natural inclusion
\[
\Omega_{G}^{tor}\hookrightarrow\Omega_{\tilde{G}}^{tor}.
\]
\end{lem}
\begin{proof}
Let the notations be as in equation \ref{eq:6-touple}. Let $X_{\circ}$
(respectively $Y_{\circ})$ be the subgroup of $X$ (respectively
$Y$) orthogonal to $\Phi$ (respectively $S$). Then $Y/Y_{\circ}\cong(X/X_{\circ})_{\sigma}/torsion$.
We have, 
\[
\Omega_{G}^{tor}\cong(Y/\mathbb{Z}E)^{tor}
\]
 and
\[
\Omega_{G^{der}}^{tor}\cong\left((Y/Y_{\circ})/\mathbb{Z}E\right)^{tor}\cong Y/(Y_{\circ}+\mathbb{Z}E).
\]
 \cite[2.15 (a)]{Springer1979}. The surjection $Y/\mathbb{Z}E\twoheadrightarrow(Y/Y_{\circ})/\mathbb{Z}E$
induces a map $\Omega_{G}^{tor}\rightarrow\Omega_{G^{der}}^{tor}$.
Let $\chi\in(Y_{\circ}+\mathbb{Z}E)$ such that $n\chi\in\mathbb{Z}E$.
Write $\chi=\chi_{\circ}+q$ where $\chi_{\circ}\in Y_{\circ}$ and
$q\in\mathbb{Z}E$. Then $n\chi\in\mathbb{Z}E$ implies $n\chi_{\circ}\in Y_{\circ}\cap\mathbb{Z}E=0$
\cite[lemma 1.2]{Springer1979}. This shows that the map $\Omega_{G}^{tor}\rightarrow\Omega_{G^{der}}^{tor}$
is injective. Since $\Omega_{G^{der}}^{tor}\hookrightarrow\Omega_{(G^{der})^{sc}}^{tor}$,
it follows that there is a natural inclusion 
\[
\Omega_{G}^{tor}\hookrightarrow\Omega_{\tilde{G}}^{tor}.
\]

\end{proof}
We then have, 

\begin{equation}
\xymatrix{\Omega_{G}^{tor}\ar@{^{(}->}[r] & \Omega_{\tilde{G}}^{tor}\\
\Omega_{G,x}^{tor}\ar@{^{(}->}[u]\ar@{^{(}->}[r] & \Omega_{\tilde{G},x}^{tor}.\ar@{^{(}->}[u]
}
\end{equation}
 It follows that

\begin{equation}
\xymatrix{\widehat{\Omega_{\tilde{G}}^{tor}}\ar@{->>}[r]\ar@{->>}[d] & \widehat{\Omega_{G}^{tor}}\ar@{->>}[d]\\
\widehat{\Omega_{\tilde{G},x}^{tor}}\ar@{->>}[r] & \widehat{\Omega_{G,x}^{tor}}.
}
\label{R-Omega commute diag}
\end{equation}

By lemma \ref{Omega X rel}, the conjugacy classes of hyperspecial
subgroups of $G$ (resp $\tilde{G}$) is a principal homogenous space
for $\widehat{\Omega_{G}^{tor}}$ (resp $\widehat{\Omega_{\tilde{G}}^{tor}}$).
Let $\tilde{R}=R(\tilde{\boldsymbol{G}},\tilde{\lambda}).$ Then by
proposition \ref{prop- Omega relates R}, $\Omega_{\tilde{G},x}^{tor}\cong\tilde{R}$
and $\Omega_{G,x}^{tor}\cong R.$ (note that $\Omega_{\tilde{G},x}^{tor}=\Omega_{\tilde{G},x}$
and $\Omega_{G,x}^{tor}=\Omega_{G,x}$) Let $\mu\mapsto\rho$ under
$\hat{\tilde{R}}\twoheadrightarrow\hat{R}$ and $\tilde{\omega}\mapsto\omega$
under $\widehat{\Omega_{\tilde{G}}^{tor}}\twoheadrightarrow\widehat{\Omega_{G}^{tor}}$.
Then from the diagram \ref{R-Omega commute diag}, it follows that
$\tilde{\omega}\cdot\mu\mapsto\omega\cdot\rho$. Using lemma \ref{lem: principal surjective.},
we get the diagram 

\begin{equation}
\xymatrix{I(\lambda)\ar@{->>}[r] & I(\tilde{\lambda})\\
\pi_{\rho}\ar@{->>}[r]\ar@{^{(}->}[u] & \underset{\mu\mapsto\rho}{\oplus}\tau_{\mu}.\ar@{^{(}->}[u]
}
\end{equation}

Let $\pi_{\rho}$ be an irreducible component of $I(\lambda)$. Then
it follows from \cite[Sec 4, Theorem]{keys1982reducibility}, that
there exits a hyperspecial vertex $h$ such that $\bar{\pi}_{\rho}^{\tilde{G}_{h}}\neq0.$
Choose $0\neq f\in I(\lambda)^{G_{h}}$ such that $f(1)=1.$ Then
$0\neq\bar{f}\in\bar{\pi}_{\rho}^{\tilde{G}_{h}}$. This implies $f\in\pi_{\rho}$
and therefore $\pi_{\rho}^{G_{h}}\neq0.$ Thus, $\pi_{\rho}^{[G_{h}]}\neq0$
iff $\bar{\pi}_{\rho}^{[\tilde{G}_{h}]}\neq0$. Equivalently,
\begin{equation}
\pi_{\omega\cdot\rho}^{\omega\cdot[G_{h}]}\neq0\mbox{ iff }\bar{\pi}_{\omega\cdot\rho}^{\tilde{\omega}\cdot[\tilde{G}_{h}]}\neq0.\label{eq: G_h 1}
\end{equation}
 $ $Also, from the statement of our theorem proved for the semisimple,
simply connected and unramified groups in section \ref{sub: G sc},
it follows that
\begin{equation}
\bar{\pi}_{\omega\cdot\rho}^{\tilde{\omega}\cdot[\tilde{G}_{h}]}\neq0\mbox{ iff }\bar{\pi}_{\rho}^{[\tilde{G}_{h}]}\neq0.\label{eq: G_h 2}
\end{equation}
From equations \ref{eq: G_h 1} and \ref{eq: G_h 2}, it follows that
for any hyperspecial subgroup $K$ of $G$, $\pi_{\rho}^{[K]}\neq0$
iff $\pi_{\omega\cdot\rho}^{\omega\cdot[K]}\neq0$.

\subsection{\label{non unitary case}Non-unitary case}

\subsubsection{Construction of the L-packet}

The following construction is given in \cite{Shahidi2011}.

Let $\varphi=\varphi_{\lambda}$ be the Langlands parameter attached
to the character $\lambda$. The parameter $\varphi$ determines a
commuting pair $\varphi_{\circ}$ and $\varphi_{+}$ such that $\varphi_{\circ}$
is tempered and 
\[
\varphi(w)=\varphi_{\circ}(w)\varphi_{+}(w)
\]
 for all $w\in W_{k}^{\prime}$ where $W_{k}^{\prime}$ is the Weil-Deligne
group. 

Let $\mu\colon T\rightarrow\mathbb{C}^{\times}$ be the character
attached to $\varphi_{\circ}$ by Langlands. Similarly let $\nu\colon T\rightarrow\mathbb{C}^{\times}$
be the character attached to $\varphi_{+}$. Let $\varpi$ be a uniformizer
in $k$. The simple roots (notations as in equation \ref{eq:6-touple})
\[
E^{\prime}=\{\alpha\in E:|\nu(\check{\alpha}(\varpi)|=1\}
\]
 generate a Levi subgroup $\boldsymbol{M}$ of $\boldsymbol{G}$ whose
dual group $\boldsymbol{\hat{M}}$ contains Im$(\varphi)$ in $\boldsymbol{\hat{G}}$.
Also, $R_{\varphi_{\circ}}(\boldsymbol{M})\cong R_{\varphi}(\boldsymbol{G})$.

Let $\boldsymbol{P=MN}$ be the parabolic subgroup of $\boldsymbol{G}$
with $\boldsymbol{M}$ as a Levi subgroup and $\boldsymbol{N}\subset\boldsymbol{U}$
where $\boldsymbol{U}$ is the unipotent radical of $\boldsymbol{B}$.
Let $\boldsymbol{U_{M}=U\cap M}$. Then 
\[
\tau=\mbox{Ind}_{TU_{M}}^{M}\mu
\]
 is a tempered representation of $M$ which may not be irreducible.
Write 
\[
\tau=\bigoplus_{i=1}^{n}\tau_{i},
\]
 where $\tau_{i}$ are irreducible. By replacing $\nu$ , $\tau_{i}$
and $\boldsymbol{M}$ with a $W(\boldsymbol{G,A})$ conjugate, we
may assume that 
\[
I(\nu,\tau_{i})=\mbox{Ind}_{P}^{G}\tau_{i}\otimes\nu
\]
 is in the Langlands setting. The elements of the $L$-packet are
then the unique Langlands quotients $J(\nu,\tau_{i})$ of each $I(\nu,\tau_{i})$
$1\leq i\leq n$. i.e., 
\[
\boldsymbol{\Pi}(\varphi_{\lambda})=\{J(\nu,\tau_{i}):1\leq i\leq n\}.
\]

\begin{lem}
There is a natural embedding $\Omega_{M}^{tor}\hookrightarrow\Omega_{G}^{tor}.$\end{lem}
\begin{proof}
We have $\Omega_{M}\cong Y/\mathbb{Z}E^{\prime}$ and $\Omega_{G}\cong Y/\mathbb{Z}E$.
So we have a natural surjection $Pr:\Omega_{M}\twoheadrightarrow\Omega_{G}$.
Let $\bar{\chi}\in\Omega_{M}^{tor}$ such that $Pr(\bar{\chi})=0$.
Then we can choose a representative $\chi\in\mathbb{Z}E$ of $\bar{\chi}$
such that $n\chi\in\mathbb{Z}E^{\prime}$ for some integer $n$. Write
\[
\chi=\sum_{i\in I}a_{i}\alpha_{i},
\]

where $a_{i}\in\mathbb{Z}$, $\alpha_{i}\in E$ and $I$ is an indexing
set for $E$. Write 
\[
n\chi=\sum_{j\in J}b_{j}\alpha_{j},
\]

where $J\subset I$ is an indexing set for $E^{\prime}$. Then 
\[
n\sum_{i\in I\smallsetminus J}a_{i}\alpha_{i}+\sum_{j\in J}(na_{j}-b_{j})\alpha_{j}=0.
\]

This implies $a_{i}=0$ for $i\in I\smallsetminus J$ and therefor
$\chi\in\mathbb{Z}E^{\prime}$, i.e., $\bar{\chi}=0$. Thus, the projection
$Pr$ induces an embedding $\Omega_{M}^{tor}\hookrightarrow\Omega_{G}^{tor}$.
\end{proof}
By the lemma, we have a surjection $\mathscr{L}\colon\widehat{\Omega_{G}^{tor}}\twoheadrightarrow\widehat{\Omega_{M}^{tor}}$.
By proposition \ref{prop- Omega relates R}, there is a natural surjection
$\widehat{\Omega_{M}^{tor}}\twoheadrightarrow\widehat{R_{\varphi_{\circ}}(\boldsymbol{M})}$.
Therefore, we have a natural surjection $\widehat{\Omega_{G}^{tor}}\twoheadrightarrow\hat{R}_{\varphi}$
from the diagram 
\begin{equation}
\xymatrix{\widehat{\Omega_{G}^{tor}}\ar@{->>}[r]\ar@{->>}[d] & \widehat{\Omega_{M}^{tor}}\ar@{->>}[d]\\
\hat{R}_{\varphi}\ar@{->}[r]^{\cong} & \widehat{R_{\varphi_{\circ}}(\boldsymbol{M})}.
}
\label{Omega diag}
\end{equation}

Let $K_{M}$ be a hyperspecial subgroup of $M$ such that $\tau_{\rho}^{K_{M}}\neq0$,
$\rho\in\hat{R}_{\varphi}$. Then there is a hyperspecial subgroup
$K_{G}$ of $G$ containing some conjugate of $K_{M}$. Withiout loss
of generality, $K_{G}\supset K_{M}$.
\begin{claim}
$I(\nu,\tau_{\rho})^{K_{G}}\neq0.$ \end{claim}
\begin{proof}
Denote by $U_{\rho}$ the space realizing $\tau_{\rho}$ and let $u$
be a $K_{M}$ fixed vector in $U_{\rho}$. Let $f:\boldsymbol{G}\rightarrow U_{\rho}$
be the function defined by 
\[
f(mnk)=\delta^{1/2}(m)\tau_{\rho}(m)u,
\]
 where $m\in M$, $n\in N$ and $k\in K_{G}$. Then $f$ is well defined
and is fixed by $K_{G}$. \end{proof}
\begin{lem}
\label{lemm J}$J(\nu,\tau_{\rho})^{K_{G}}\neq0.$ \end{lem}
\begin{proof}
The Langlangs quotient $J(\nu,\tau_{\rho})$ is the image of an intertwining
operator 
\[
A(\nu,\tau_{\rho,}w)\colon I(\nu,\tau_{\rho})\rightarrow I(w\nu,w\tau_{\rho}),
\]
 for some $w\in W$. Then $J(\nu,\tau_{\rho})^{K_{G}}\neq0$ follows
from the fact that the image of a $K_{G}$ spherical vector in $I(\lambda)$
under 
\[
A(\lambda,w)\colon I(\lambda)\rightarrow I(w\lambda)
\]
is non-zero. 
\end{proof}
Conversely, if $J(\nu,\tau_{\rho})^{K_{G}}\neq0$, then some conjugate
of $K_{G}$ contains a hyperspecial $K_{M}$ of $M$. Without loss
of generality, $K_{G}\supset K_{M}$. If $0\neq g\in I(\nu,\tau_{\rho})^{K_{G}}$,
then $0\neq g(1)\in\tau_{\rho}^{K_{M}}$.

In the notations as above, let $\pi_{\rho}=J(\nu,\tau_{\rho}).$ Let
$\omega\in\widehat{\Omega_{G}^{tor}}$ and $\bar{\omega}\in\widehat{\Omega_{M}^{tor}}$
be its image under $\mathscr{L}:\omega\mapsto\bar{\omega}.$ Then
it follows from lemma \ref{lemm J} and the diagram \ref{Omega diag},

\begin{eqnarray*}
\pi_{\omega\cdot\rho}^{\omega\cdot[K_{G}]}\neq0 & \iff & \tau_{\bar{\omega}\cdot\rho}^{\bar{\omega}\cdot[K_{M}]}\neq0\\
 & \iff & \tau_{\rho}^{[K_{M}]}\neq0\\
 & \iff & \pi_{\rho}^{[K_{G}]}\neq0.
\end{eqnarray*}

This completes the proof of the main theorem.

\section{\label{main prop proof}Proof of Proposition \ref{intertwining coeff}
and \ref{intertwining coeff-1}}

\subsection{Affine Root Structure}

Let $\boldsymbol{G}$ be semisimple, almost simple, simply connected
and unramified. Let $\boldsymbol{N}=N_{\boldsymbol{G}}\boldsymbol{A}$
be the normalizer of $\boldsymbol{A}$ and let $W$ be the relative
Weyl group of $\boldsymbol{G}$. Let $S$ denote the set of relative
roots. Bruhat and Tits define \emph{affine roots} $\Sigma$ associated
to $G.$ In the split case or in the case that $\boldsymbol{G}$ is
unramified and of type $^{2}A_{2n-1}(n\geqslant3)$ or $^{2}D_{n+1}(n\geq2)$,
\[
\Sigma=\{a+n:a\in S,n\in\mathbb{Z}\}.
\]

Affine roots can be regarded as affine linear functions $x\mapsto a(x)+n$
on the vector space spanned by $S$. Let $h_{\alpha}$ be the hyperplane
on which $\alpha\in\Sigma$ vanishes. Let $s_{\alpha}$ be the reflection
in the hyperplane $h_{\alpha}$. The affine Weyl group $\tilde{W}$
is generated by the reflections $s_{\alpha}$, $\alpha\in\Sigma$.
For $\alpha\in\Sigma$, let $t_{\alpha}$ be the translation $s_{\alpha}s_{\alpha+1}$.
There exists a canonical surjection $\nu:N\rightarrow\tilde{W}$.
Let $H$ be the kernel of $\nu$. 

There exists a collection of compact unipotent subgroups $U_{\alpha}$
of $G$, $\alpha\in\Sigma,$ satisfying the following properties (see
\cite[page 27]{MR0435301}):
\begin{enumerate}
\item $xU_{\alpha}x^{-1}=U_{\nu(x)\alpha}$ for $x\in N$, $\alpha\in\Sigma$.
\item $U_{\alpha}\subsetneq U_{\alpha-1}$, $q_{\alpha}=(U_{\alpha-1}:U_{\alpha})$
is finite and $\underset{k\in\mathbb{Z}}{\bigcap}U_{\alpha+k}=\{1\}$.
\item If $\beta=-\alpha+r$, $r\gneq0$, then $<U_{\alpha},U_{\beta},H>=U_{\alpha}HU_{\beta}$
.
\item $<U_{\alpha},U_{\beta},H>=U_{\alpha}\nu^{-1}(s_{\alpha})U_{\alpha}\cup U_{\alpha}HU_{-\alpha+1}$.
\item If $h_{\alpha}$ and $h_{\beta}$ are not parallel, i.e., if $\beta\neq\pm\alpha+r$
for any $r\in\mathbb{Z}$, then the commutator group $[U_{\alpha},U_{\beta}]$
is contained in $<U_{r\alpha+s\beta}|r\alpha+s\beta\in\Sigma,\, r,s\geq0>.$ 
\item Let $S^{+}$ and $S^{-}$ denote the positive and negative roots respectively
in $S$. For each $\alpha\in S$, let $U_{(\alpha)}=\underset{r\in\mathbb{Z}}{\bigcup}U_{\alpha+r}$.
Let $U^{+}=<U_{(\alpha)}:\alpha\in S^{+}>$ and $U^{-}=<U_{(\alpha)}:\alpha\in S^{-}>$.
Then $U^{+}\cap MU^{-}=\{1\}.$
\item $G=<N,U_{\alpha}:\alpha\in\Sigma>$.\end{enumerate}
\begin{lem}
\label{keys lemma}\cite[Lemma]{Keys1982} Let $\alpha$ be a positive
simple root and let $f_{K,\lambda}\in I(\lambda)$ be the $K$-fixed
vector with $f_{K,\lambda}(1)=1$. Then $\mathscr{A}(s_{\alpha},\lambda)f_{K,\lambda}=f_{K,s_{\alpha}\lambda}$
if $U_{(\alpha)}\cap K=U_{\alpha}$ and $\mathscr{A}(s_{\alpha},\lambda)f_{K,\lambda}=\lambda(t_{\alpha})^{-1}f_{K,s_{\alpha}\lambda}$
if $U_{(\alpha)}\cap K=U_{\alpha-1}$ and $q_{\alpha/2}=1.$ 
\end{lem}

\subsection{More about the group $\Omega$ }

Let $\boldsymbol{G}$ be semisimple, almost simple, simply connected
and unramified. Let $\underbar{\ensuremath{\Psi}}=(Y,S,E,\check{Y},\check{S},\check{E})$
be the based root datum defined in section \ref{Relative root system}.
$W$ is the relative Weyl group. Let $\check{E}=\{\check{\beta_{1}},\ldots,\check{\beta_{n}}\}$
where $\check{\beta_{i}},$ $1\leq i\leq n$, are the simple relative
co-roots. Let $\check{\tilde{\beta_{0}}}=\sum_{i=1}^{n}n_{i}\check{\beta_{i}}$
be the highest co-root. Put $\check{\beta_{0}}=1-\check{\tilde{\beta_{0}}}$.
Let $h=\sum_{i=1}^{n}n_{i}+1$ be the coxeter number and put $\delta_{\circ}=\frac{1}{2}\sum_{\beta\in S^{+}}\beta$.
Recall that $\underbar{C}=C(\check{\underbar{\ensuremath{\Psi}}})$
was the alcove in the apartment $A=Y\otimes\mathbb{R}$ defined by
$\{x\in A:\check{\beta_{0}}(x)\geq0,\ldots,\check{\beta}_{n}(x)\geq0\}$.
Put $n_{0}=1$ and give the weight $n_{i}$ to the $i^{th}$ vertex
of $\underbar{C}$. Then $c_{\circ}=h^{-1}\delta_{\circ}$ is the
(weighted) barycenter of $\underbar{C}$. For any $w\in W$, let $\tilde{w}$
be the affine map $x\in A\mapsto w(x-c_{\circ})+c_{\circ}$. It is
the unique affine map fixing $c_{\circ}$ with tangent part $w$.
Denote by $Q$ the lattice generated by $S$. 
\begin{lem}
\label{jkyu lemma}\cite[Lemma 6.2]{Jiu-KangYu} For any $w\in W$,
the following are equivalent :\end{lem}
\begin{enumerate}
\item $\tilde{w}\in W\ltimes Y.$
\item $\tilde{w}\underbar{C}=\underbar{C}.$ 
\item $\tilde{w}(\check{E_{\circ}})=\check{E_{\circ}}$ where $\check{E_{\circ}}=\check{E}\cup\{\check{\beta_{0}}\}.$ 
\item $w(\check{E}\cup\{-\check{\tilde{\beta_{0}}}\})=\check{E}\cup\{-\check{\tilde{\beta_{0}}}\}$.

Then $\underbar{\ensuremath{\Omega}}=\Omega(\check{\underbar{\ensuremath{\Psi}}})$
is precisely the set of those $w\in W$ satisfying these conditions
and $w\mapsto\tilde{w}(\check{\beta_{0}})$ is a bijection between
$\underbar{\ensuremath{\Omega}}$ and $\{\check{\beta_{i}}\in\check{E_{\circ}}:n_{i}=1\}$.
There is an isomorphism $\iota:\underbar{\ensuremath{\Omega}}\rightarrow Y/Q$
defined by any of the following ways:

\item[(a)] $\iota(w)=(1-w)h^{-1}\delta+Q.$

\item[(b)] $\iota(w)$ is the image of $\tilde{w}$ under $W\ltimes Y\rightarrow(W\ltimes Y)/(W\ltimes Q)=Y/Q.$

\item[(c)] If $w\neq0$ and $\tilde{w}(\check{\beta_{0}})=\check{\beta_{i}}$,
$\iota(w)$ is the $i^{th}$ fundamental weight $\omega_{i}$ modulo
$Q$. For $w=1$, $\iota(w)=0$. 

\end{enumerate}

\subsection{\label{calculations}Calculations}

Let the notations be as in section \ref{Relative root system}. Further,
$\boldsymbol{G}$ is semisimple, almost simple, simply connected and
unramified. Let $K_{0}$ be a hyperspecial subgroup satisfying $\pi_{0}^{K_{0}}\neq0.$
Then $K_{0}$ determines the choice of an origin in the Bruhat-Tits
building of $G.$ For $K_{\omega}\in\omega\cdot[K_{0}],$ let $f_{\omega,\lambda}\in I(\lambda)^{K_{\omega}}$
be the spherical vector such that $f_{\omega,\lambda}(1)=1$. For
$w\in W,$ let $c_{\lambda}(\omega,w)$ be the coefficient of $f_{\omega,w\cdot\lambda}$
in the equation: $\mathscr{A}(w,\lambda)f_{\omega,\lambda}=c_{\lambda}(\omega,w)f_{\omega,w\cdot\lambda}$.
Then $c_{\lambda}(\omega,w)$ satisfies the co-cycle relation, 
\begin{equation}
c_{\lambda}(\omega,w_{1}w_{2})=c_{w_{2}\cdot\lambda}(\omega,w_{1})\cdot c_{\lambda}(\omega,w_{2}).\label{coeff co-cycle relation}
\end{equation}

Let $x_{\lambda}$ denote the point in $\underbar{C}$ associated
to $\lambda$. Define $\tilde{c}_{\lambda}(\omega,w)\in\mathbb{Q}/\mathbb{Z}$
by $c_{\lambda}(\omega,w)=e^{-2\pi i\tilde{c}_{\lambda}(\omega,w)}$.
Then $\tilde{c}(\omega,w)$ satisfies the cocycle relation:
\begin{equation}
\tilde{c}_{\lambda}(\omega,w_{1}w_{2})=\tilde{c}_{w_{2}\cdot\lambda}(\omega,w_{1})+\tilde{c}_{\lambda}(\omega,w_{2}).\label{eq:cocycle-1}
\end{equation}
Let $\mathcal{C}=\{x:x\mbox{ is a hyperspecial vertex in the closure of }\underbar{C}\}.$
Then the representatives of the conjugacy classes of hyperspecial
subgroups can be chosen to be $\{G_{x}:x\in\mathcal{C}\}.$ Also,
the action of $\hat{\underbar{\ensuremath{\Omega}}}$ on the conjugacy
classes of hyperspecial subgroups gives a bijection $\hat{\underbar{\ensuremath{\Omega}}}\rightarrow\mathcal{C}.$
We use this bijection to identify $\hat{\underbar{\ensuremath{\Omega}}}$
with $\mathcal{C}$ whenever there is no ambiguity. By lemma \ref{keys lemma},
we have, 
\begin{eqnarray*}
\mathscr{A}(s_{\alpha},\lambda)f_{\omega,\lambda} & = & \begin{cases}
\lambda(t_{\alpha})^{-1}f_{\omega,s_{\alpha}\lambda} & \alpha(\omega)=1\\
1 & otherwise.
\end{cases}
\end{eqnarray*}
 Therefore,

\[
c_{\lambda}(\omega,s_{\alpha})=\begin{cases}
e^{-2\pi i(x_{\lambda},\check{\alpha})} & \alpha(\omega)=1\\
1 & otherwise.
\end{cases}
\]

Here $(x_{\lambda},\check{\alpha})$ denotes $\check{\alpha}(x_{\lambda}).$
Then,
\begin{equation}
\tilde{c}_{\lambda}(\omega,s_{\alpha})=\begin{cases}
(x_{\lambda},\check{\alpha}) & \alpha(\omega)=1\\
0 & otherwise.
\end{cases}\label{eq: keys coeff}
\end{equation}

Using proposition \ref{prop- Omega relates R}, we identify $R_{\lambda}$
as a subgroup of $\underbar{\ensuremath{\Omega}}$. For $r\in R_{\lambda}$,
we calculate $\tilde{c}_{\lambda}(\omega,r)$ by using equations \ref{coeff co-cycle relation}
and \ref{eq: keys coeff}. Using lemma \ref{jkyu lemma}, we realize
$R_{\lambda}$ as a subgroup of $Y/Q$. Finally, we calculate $(\omega,r)$
and show that $\tilde{c}_{\lambda}(\omega,r)=(\omega,r)\mbox{ mod }\mathbb{Z}$,
implying $c_{\lambda}(\omega,r)=\rho_{\omega}(r)$. 

When $\boldsymbol{G}$ is split, $\underbar{\ensuremath{\Omega}}$
is non-trivial for the cases $A_{n},B_{n},C_{n},D_{n},E_{6}$ and
$E_{7}$. When $\boldsymbol{G}$ is unramified but not split, $\underbar{\ensuremath{\Omega}}$
is non-trivial only for the cases $^{2}A_{2n-1}(n\ge3)$ and $^{2}D_{n+1}(n\geq2)$
\cite[table-1, page 29]{MR2674853}. The method of calculation of
$c_{\lambda}(\omega,w)$ is the same in all cases. We illustrate the
cases below:

\subsubsection{Type $A_{n}$ }

Let $\{\epsilon_{1},\epsilon_{2},\ldots,\epsilon_{n+1}\}$ be the
standard orthonormal basis of $\mathbb{R}^{n+1}$. We realize $A_{n}$
as a root system in the vector space $\mathbb{E}$ where, 
\[
\mathbb{E}=\{\sum_{i=1}^{n+1}c_{i}\epsilon_{i}:c_{i}\in\mathbb{R},\sum c_{i}=0\}.
\]
Then $\Phi=\{\epsilon_{i}-\epsilon_{j}:i\neq j,0\leq i,j\leq n\}$
is the set of roots. Let $\alpha_{i}=\epsilon_{i}-\epsilon_{i+1}$
for $1\leq i\leq n$. The set $\Delta=\{\alpha_{1},\ldots,\alpha_{n}\}$
is a fundamental system in $\Phi$. Also, $\check{\Phi}=\Phi$ and
$\check{\Delta}=\Delta$. The highest root is $\beta=\alpha_{1}+\ldots+\alpha_{n}=\epsilon_{1}-\epsilon_{n+1}$.
The fundamental weights are $\{x_{1},\ldots,x_{n}\}$, where 
\[
x_{k}=\epsilon_{1}+\epsilon_{2}+\ldots\epsilon_{k}-\frac{k}{n+1}(\epsilon_{1}+\epsilon_{2}+\ldots+\epsilon_{n+1}).
\]
 Let $\tilde{w_{0}}$ be the generator of the $\Omega(\check{\Psi})$
and let $w_{0}$ be the vector part. Then $w_{0}\cdot\alpha_{i}=\alpha_{i+1}$
for $1\leq i\leq n-1$ and $w_{0}\cdot\alpha_{n}=-\beta$, by lemma
\ref{jkyu lemma}. Write $s_{i}$ for the reflection about the hyperplane
determined by the root $\alpha_{i}$. As a product of reflections,
$w_{0}=(12\ldots n+1)=s_{1}s_{2}\ldots s_{n}$.Write 
\begin{eqnarray*}
x_{\lambda} & = & \sum_{i=1}^{n}a_{i}x_{i}\\
 & = & (b_{1},b_{2},\ldots,b_{n+1})
\end{eqnarray*}
 for some $a_{i}\in\mathbb{Q}$ and where $b_{i}=\frac{1}{(n+1)}(-a_{1}-2a_{2}-\ldots-(i-1)a_{i-1}+(n+1-i)a_{i}+\ldots+(n+1-k)a_{k}+\ldots+na_{n}).$ 

If $r\in R_{\lambda}$, then $r=w_{0}^{d}$ for some integer $d$.
Then by the cocycle relation \ref{eq:cocycle-1}, we get 
\begin{equation}
\tilde{c}_{\lambda}(\omega,w_{0}^{d})=\tilde{c}_{\lambda}(\omega,w_{0})+\ldots+\tilde{c}_{w_{0}^{d-2}\lambda}(\omega,w_{0})+\tilde{c}_{w_{0}^{d-1}\lambda}(\omega,w_{0}).\label{eq:An co-cycle}
\end{equation}
Let $\omega_{i}$ be the $i^{th}$ fundamental co-weight, i.e., $\alpha_{j}(\omega_{i})=\delta_{ij}$.
The set of hyperspecial verticles in the closure of the fundamental
alcove is $\mathcal{C}=\{0,\omega_{i}\mbox{ for }1\leq i\leq n\}$.
$ $Using equations \ref{eq:cocycle-1} and \ref{eq: keys coeff},
we get,
\[
\tilde{c}_{\lambda}(\omega,w_{0})=b_{i}-b_{n+1}.
\]
 From this and the equation \ref{eq:An co-cycle}, we get

\begin{equation}
\tilde{c}_{\lambda}(\omega_{i},w_{0}^{d})=(b_{i}-b_{n+1})+(b_{i-1}-b_{n})+\ldots.\label{eq:coeff}
\end{equation}
 Since $\tilde{w_{0}}^{d}\cdot x_{\lambda}=x_{\lambda}$, it follows
that $a_{i}=a_{j}$ iff $i\equiv j$ mod $d$. From this and the last
equation, it follow that 
\begin{eqnarray*}
\tilde{c}_{\lambda}(\omega_{i},w_{0}^{d}) & = & (a_{0}+\ldots+a_{d-1})(n+1-i)\\
 & = & \frac{(n+1-i)d}{n+1}.
\end{eqnarray*}
 The last equality follows because $\sum_{i=0}^{n}a_{i}=1$. Now $\tilde{w}_{0}^{d}\cdot x_{0}=x_{d}$,
where $x_{0}$ is the origin in $V(\check{\Psi})$. Therefore, 
\[
(\omega_{i},r)=(\omega_{i},x_{d})=\frac{(n+1-i)d}{n+1}.
\]
 Thus, 
\[
\tilde{c}_{\lambda}(\omega_{i},r)=(\omega_{i},r)\mbox{ mod }\mathbb{Z}.
\]
This prooves the result for type $A_{n}$.

\subsubsection{Type $B_{n}$ }

The dual root sysyem of $B_{n}$ is $C_{n}$. Let $\{\epsilon_{1},\epsilon_{2},\ldots,\epsilon_{n}\}$
be the standard orthonormal basis of $\mathbb{R}^{n}$. Then $B_{n}$
root system is given by $\Phi=\{\pm\epsilon_{i}\pm\epsilon_{j}\}\cup\{\pm\epsilon_{i}\}$.
The co-roots are $\check{\Phi}=\{\pm\epsilon_{i}\pm\epsilon_{j}\}\cup\{\pm2\epsilon_{i}\}$.
Let $\alpha_{i}=\epsilon_{i}-\epsilon_{i+1}$ for $1\leq i\leq n-1$
and $\alpha_{n}=\epsilon_{n}$. The set $\Delta=\{\alpha_{1},\ldots,\alpha_{n}\}$
is a set of simple roots in $\Phi$. The corresponding simple co-roots
are $\check{\Delta}=\{\check{\alpha}_{1},\ldots,\check{\alpha}_{n}\}$
where $\check{\alpha_{i}}=\alpha_{i}$ for $1\leq i\leq n-1$ and
$\check{\alpha_{n}}=2\alpha_{n}$. The highest co-root is $\check{\beta}=2(\check{\alpha}_{1}+\dots+\check{\alpha}_{n-1})+\check{\alpha}_{n}$.
The fundamental weights are $\{x_{1},\ldots,x_{n}\}$ where,

\[
x_{k}=\epsilon_{1}+\ldots+\epsilon_{k}\mbox{ for \ensuremath{1\leq k\leq n-1}},
\]
\[
x_{n}=\frac{1}{2}(\epsilon_{1}+\ldots+\epsilon_{n-1}+\epsilon_{n}).
\]
Let $\tilde{w_{0}}$ be the generator of the $\Omega(\check{\Psi})$
and let $w_{0}$ be the vector part. Then $w_{0}\cdot\check{\alpha}_{i}=\check{\alpha}_{n-i}$
for $1\leq i\leq n-1$ and $w_{0}\cdot\check{\alpha}_{n}=-\check{\beta}$.
Thus for $y=(y_{1},\ldots,y_{n})\in\mathbb{R}^{n}$, $w_{0}\cdot y=(-y_{n},-y_{n-1},\ldots,-y_{1})$.
As a product of reflections, \linebreak{}
$w_{0}=\dots qs_{n-2}s_{n-1}qs_{n-1}qq$, where $q=s_{n}\dots s_{1}$.
Write 
\begin{eqnarray*}
x_{\lambda} & = & \sum_{i=1}^{n}a_{i}x_{i}\\
 & = & (b_{1},b_{2},\ldots,b_{n}),
\end{eqnarray*}
 for some $a_{i}\in\mathbb{Q}$ and where $b_{i}=a_{i}+\ldots+a_{n-1}+\frac{1}{2}a_{n}$.
Let $\omega_{1}$ be the first fundamental co-weight. Then $\omega_{1}=\epsilon_{1}$
and $\mathcal{C}=\{0,\omega_{1}\}$. Using equations \ref{eq:cocycle-1}
and \ref{eq: keys coeff} to calculate $\tilde{c}(\omega_{1},w_{0})$
we get,
\[
\tilde{c}(\omega_{1},w_{0})=(h,\check{\alpha}_{1})
\]
 where $h=x_{\lambda}+q\cdot x_{\lambda}+s_{n-1}qq\cdot x_{\lambda}+\ldots$.
Thus 
\begin{eqnarray}
\tilde{c}(\omega_{1},w_{0}) & = & (h,\epsilon_{1}-\epsilon_{2})\nonumber \\
 & = & b_{1}+b_{n}\nonumber \\
 & = & a_{1}+\ldots+a_{n-1}+\frac{1}{2}a_{n}+\frac{1}{2}a_{n}\nonumber \\
 & = & a_{1}+\ldots+a_{n}.\label{coff B_n}
\end{eqnarray}
 Since $\tilde{w_{0}}\cdot x_{\lambda}=x_{\lambda}$, it follows that
$a_{i}=a_{n-i}$ for $1\leq i\leq n-1$ and $a_{0}=a_{n}$. From this
and the relation
\[
a_{0}+2(a_{1}+\ldots+a_{n-1})+a_{n}=1,
\]
 it follows that $\tilde{c}(\omega_{1},w_{0})=\frac{1}{2}$. Also
$(\omega_{1},x_{n})=\frac{1}{2}$. This completes the case of $B_{n}$.

\subsubsection{Type $C_{n}$ }

Then $C_{n}$ root system is given by $\Phi=\{\pm\epsilon_{i}\pm\epsilon_{j}\}\cup\{\pm2\epsilon_{i}\}$.
Let $\alpha_{i}=\epsilon_{i}-\epsilon_{i+1}$ for $1\leq i\leq n-1$
and $\alpha_{n}=2\epsilon_{n}$. The set $\Delta=\{\alpha_{1},\ldots,\alpha_{n}\}$
is a set of simple roots in $\Phi$. The corresponding simple co-roots
are $\check{\Delta}=\{\check{\alpha}_{1},\ldots,\check{\alpha}_{n}\}$,
where $\check{\alpha}_{i}=\alpha_{i}$ for $1\leq i\leq n-1$ and
$\check{\alpha}_{n}=\frac{1}{2}\alpha_{n}$. The highest co-root is
$\check{\beta}=\check{\alpha}_{1}+2(\check{\alpha}_{2}+\dots+\check{\alpha}_{n})=\epsilon_{1}+\epsilon_{2}$.
The fundamental weigts are $\{x_{1},\ldots,x_{n}\}$ where, 
\[
x_{k}=\epsilon_{1}+\ldots+\epsilon_{k}\mbox{ for \ensuremath{1\leq k\leq n}}.
\]

Let $\tilde{w_{0}}$ be the generator of the $\Omega(\check{\Psi})$
and let $w_{0}$ be the vector part. Then $w_{0}\cdot\check{\alpha}_{i}=\check{\alpha}_{i}$
for $2\leq i\leq n$ and $w_{0}\cdot\check{\alpha}_{1}=-\check{\beta}$.
Thus for $y=(y_{1},\ldots,y_{n})\in\mathbb{R}^{n}$, $w_{0}\cdot y=(-y_{1,}y_{2},\ldots,y_{n})$.
As a product of reflections, $w_{0}=s_{1}\ldots s_{n-1}s_{n}s_{n-1}\ldots s_{1}$.
Write
\begin{eqnarray*}
x_{\lambda} & = & \sum_{i=1}^{n}a_{i}x_{i}\\
 & = & (b_{1},b_{2},\ldots,b_{n}),
\end{eqnarray*}
 for some $a_{i}\in\mathbb{Q}$ and where $b_{i}=a_{i}+\ldots+a_{n-1}+a_{n}$.
Let $\omega_{n}$ be the $n^{th}$ fundamental co-weight. Then $\omega_{n}=\frac{1}{2}(\epsilon_{1}+\ldots+\epsilon_{n-1}+\epsilon_{n})$
and $\mathcal{C}=\{0,\omega_{n}\}$. Using equations \ref{eq:cocycle-1}
and \ref{eq: keys coeff} to calculate $\tilde{c}(\omega_{n},w_{0})$,
we get,
\begin{eqnarray*}
\tilde{c}(\omega_{n},w_{0}) & = & (s_{n-1}\ldots s_{1}\cdot x_{\lambda},\check{\alpha}_{n})\\
 & = & a_{1}+\ldots+a_{n}.
\end{eqnarray*}
 Since $\tilde{w_{0}}\cdot x_{\lambda}=x_{\lambda}$, it follows that
$a_{0}=a_{1}$. From this and the relation
\[
a_{0}+a_{1}+2(a_{2}+\ldots+a_{n})=1,
\]
 it follow that $\tilde{c}(\omega_{n},w_{0})=\frac{1}{2}$. Also $(\omega_{n},x_{1})=\frac{1}{2}$.
This completes the case of $C_{n}$.

\subsubsection{Type $D_{n}$}
\begin{description}
\item [{Case}] $n$ is odd 
\end{description}
If $\{\epsilon_{1},\epsilon_{2},\ldots,\epsilon_{n}\}$ is the standard
orthonormal basis of $\mathbb{R}^{n}$, then the root system of $D_{n}$
is given by $\Phi=\{\pm\epsilon_{i}\pm\epsilon_{j}\}$. Let $\alpha_{i}=\epsilon_{i}-\epsilon_{i+1}$
for $1\leq i\leq n-1$ and $\alpha_{n}=\epsilon_{n-1}+\epsilon_{n}$.
The set $\Delta=\{\alpha_{1},\ldots,\alpha_{n}\}$ is a set of simple
roots in $\Phi$. We have $\check{\Phi}=\Phi$ and $\check{\Delta}=\Delta$.
The highest root is $\beta=\alpha_{1}+2(\alpha_{2}+\ldots+\alpha_{n-2})+\alpha_{n-1}+\alpha_{n}=\epsilon_{1}+\epsilon_{2}$.
The fundamental weights are
\[
x_{k}=\epsilon_{1}+\ldots+\epsilon_{k}\mbox{ for \ensuremath{1\leq k\leq n-2}},
\]
\[
x_{n-1}=\frac{1}{2}(\epsilon_{1}+\ldots+\epsilon_{n-1}-\epsilon_{n}),
\]
\[
x_{n}=\frac{1}{2}(\epsilon_{1}+\ldots+\epsilon_{n-1}+\epsilon_{n}).
\]
 Let $\tilde{w_{0}}$ be the generator of the $\Omega(\check{\Psi})$
and let $w_{0}$ be the vector part. Then $w_{0}\cdot\alpha_{i}=\alpha_{n-i}$
for $2\leq i\leq n-2$, $w_{0}\cdot\alpha_{1}=\alpha_{n-1}$, $w_{0}\cdot\alpha_{n-1}=-\beta$,
$w_{0}\cdot(-\beta)=\alpha_{n}$ and $w_{0}\cdot\alpha_{n}=\alpha_{1}$.
Thus for $y=(y_{1},\ldots,y_{n})\in\mathbb{R}^{n}$, $w_{0}\cdot y=(y_{n},-y_{n-1},-y_{n-2},\ldots,-y_{1})$.
As a product of reflections, \linebreak{}
$w_{0}=\ldots s_{4}\ldots s_{n-2}s_{n}s_{3}\ldots s_{n-1}s_{2}\ldots s_{n-2}s_{n}s_{1}\ldots s_{n-1}$.
Write
\begin{eqnarray*}
x_{\lambda} & = & \sum_{i=1}^{n}a_{i}x_{i}\\
 & = & (b_{1},b_{2},\ldots,b_{n})
\end{eqnarray*}
 for some $a_{i}\in\mathbb{Q}$ and where $b_{i}=a_{i}+\ldots+a_{n-2}+\frac{a_{n-1}+a_{n}}{2}$
for $1\leq i\leq n-1$ and $b_{n}=\frac{-a_{n-1}+a_{n}}{2}$. Let
$\omega_{i}$ be the $i^{th}$ fundamental co-weight. Then $\mathcal{C}=\{0,\omega_{1},\omega_{n-1},\omega_{n}\}$.
The first fundamental co-weight $\omega_{1}=x_{1}$. Using equations
\ref{eq:cocycle-1} and \ref{eq: keys coeff} to calculate $\tilde{c}(\omega_{1},w_{0})$,
we get,
\begin{eqnarray*}
\tilde{c}(\omega_{1},w_{0}) & = & (s_{2}\ldots s_{n-1}\cdot x_{\lambda},\check{\alpha}_{1})\\
 & = & b_{1}-b_{n}\\
 & = & a_{1}+\ldots+a_{n-1}.
\end{eqnarray*}
 Since $\tilde{w_{0}}\cdot x_{\lambda}=x_{\lambda}$, it follows that
$a_{i}=a_{n-i}$ for $2\leq i\leq n-2$ and $a_{0}=a_{1}=a_{n-1}=a_{n}$.
From this and the relation
\begin{equation}
a_{0}+a_{1}+2(a_{2}+\ldots+a_{n-2})+a_{n-1}+a_{n}=1\label{a_i rel D_n odd}
\end{equation}
 it follow that $\tilde{c}(\omega_{1},w_{0})=\frac{1}{2}$. Also $(\omega_{1},x_{n})=\frac{1}{2}$. 

To compute $\tilde{c}(\omega_{1},w_{0}^{2})$ and $\tilde{c}(\omega_{1},w_{0}^{3})$
we can use the co-cycle relation \ref{eq:cocycle-1}. 

For $r=w_{0}^{2}\in R_{\lambda}$ we get

\begin{eqnarray*}
\tilde{c_{\lambda}}(\omega_{1},w_{0}^{2}) & = & \tilde{c}_{\lambda}(\omega_{1},w_{0})+\tilde{c}_{w_{0}\cdot\lambda}(\omega_{1},w_{0})\\
 & = & b_{1}-b_{n}+b_{n}+b_{1}\\
 & = & 2b_{1}.
\end{eqnarray*}
 Now if $\tilde{w_{0}}^{2}\cdot x_{\lambda}=x_{\lambda}$, then it
follows $a_{0}=a_{1}$ and $a_{n-1}=a_{n}$. From this and the relation
\ref{a_i rel D_n odd} we get $b_{1}=\frac{1}{2}$ and therefore $\tilde{c_{\lambda}}(\omega_{1},w_{0}^{2})=1$.
Also $(\omega_{1},x_{1})=1$. 

For $r=w_{0}^{3}\in R_{\lambda}$ we get,

\begin{eqnarray*}
\tilde{c_{\lambda}}(\omega_{1},w_{0}^{3}) & = & \tilde{c}_{\lambda}(\omega_{1},w_{0}^{2})+\tilde{c}_{w_{0}^{2}\cdot\lambda}(\omega_{1},w_{0})\\
 & = & 2b_{1}+b_{n}-b_{1}\\
 & = & b_{n}+b_{1}.
\end{eqnarray*}
 Now if $\tilde{w_{0}}^{3}\cdot x_{\lambda}=x_{\lambda}$, then it
follows $a_{0}=a_{n-1}=a_{n-1}=a_{n}$ and $a_{i}=a_{n-i}$ for $2\leq i\leq n-2.$
From this and the relation \ref{a_i rel D_n odd} we get $b_{1}+b_{n}=\frac{1}{2}$
and therefore $\tilde{c_{\lambda}}(\omega_{1},w_{0}^{3})=\frac{1}{2}$.
Also $(\omega_{n-1},x_{1})=\frac{1}{2}$. 

The calculation for other hyperspecial vertices follows by the same
method. This completes the case of $D_{n}$ when $n$ is odd. 
\begin{description}
\item [{Case}] $n$ is even
\end{description}
In this case $\Omega(\check{\Psi})$ is not cyclic. The non-trivial
automorphisms of the extended Dynkin diagram are:
\begin{itemize}
\item $w_{1}:\{-\beta\}\cup\Delta\rightarrow\{-\beta\}\cup\Delta$
\begin{eqnarray*}
-\beta & \longleftrightarrow & \alpha_{n},\\
\alpha_{1} & \longleftrightarrow & \alpha_{n-1},\\
\alpha_{i} & \longleftrightarrow & \alpha_{n-i}.
\end{eqnarray*}
 
\item $w_{2}:\{-\beta\}\cup\Delta\rightarrow\{-\beta\}\cup\Delta$ 
\begin{eqnarray*}
-\beta & \longleftrightarrow & \alpha_{1},\\
\alpha_{n-1} & \longleftrightarrow & \alpha_{n},\\
\alpha_{i} & \longleftrightarrow & \alpha_{i}.
\end{eqnarray*}
 
\item $w_{3}:\{-\beta\}\cup\Delta\rightarrow\{-\beta\}\cup\Delta$ 
\begin{eqnarray*}
-\beta & \longleftrightarrow & \alpha_{n-1},\\
\alpha_{1} & \longleftrightarrow & \alpha_{n},\\
\alpha_{i} & \longleftrightarrow & \alpha_{n-i}.
\end{eqnarray*}
 For the fundamental weight $\omega_{1}$, we calculate $\tilde{c_{\lambda}}(\omega_{1},w)$
for each $w\in\{w_{1},w_{2},w_{3}\}$ and verify our result. 
\end{itemize}
$w_{1}:\mathbb{R}^{n}\rightarrow\mathbb{R}^{n}$ is the map $(y_{1},\ldots,y_{n})\longmapsto(-y_{n},-y_{n-1},\ldots,-y_{1})$.
As a product of reflections $w_{1}=\cdots s_{4}\cdots s_{n-1}s_{3}\cdots s_{n-2}s_{n}s_{2}\cdots s_{n-1}\cdots s_{1}\cdots s_{n-2}s_{n}$.
Using equations \ref{eq:cocycle-1} and \ref{eq: keys coeff} to calculate
$\tilde{c}(\omega_{1},w_{1})$ we get, 
\begin{eqnarray*}
\tilde{c}(\omega_{1},w_{1}) & = & (s_{2}\ldots s_{n-2}s_{n}\cdot x_{\lambda},\check{\alpha}_{1})\\
 & = & ((b_{1},-b_{n},b_{2},\ldots,b_{n-2},-b_{n-1}),\epsilon_{1}-\epsilon_{2})\\
 & = & b_{1}+b_{n}\\
 & = & a_{1}+\cdots+a_{n-1}.
\end{eqnarray*}
 Since $w_{1}\cdot x_{\lambda}=x_{\lambda}$, it follows $a_{0}=a_{n}$
and $a_{1}=a_{n-1}$. From this and the relation \ref{a_i rel D_n odd},
we get that $\tilde{c}(\omega_{1},w_{1})=\frac{1}{2}$. Also, $(w_{1},x_{n})=\frac{1}{2}$. 

$w_{2}:\mathbb{R}^{n}\rightarrow\mathbb{R}^{n}$ is the map $(y_{1},\ldots,y_{n})\longmapsto(-y_{1},y_{2},\ldots,y_{n-1},-y_{n})$.
As a product of reflections $w_{2}=s_{1}\cdots s_{n-2}s_{n}s_{n-1}\cdots s_{1}$.
Using equations \ref{eq:cocycle-1} and \ref{eq: keys coeff} to calculate
$\tilde{c}(\omega_{1},w_{2})$ we get, 
\begin{eqnarray*}
\tilde{c}(\omega_{1},w_{2}) & = & (x_{\lambda}+s_{2}\cdots s_{n-2}s_{n}s_{n-1}\cdots s_{1}\cdot x_{\lambda},\check{\alpha}_{1})\\
 & = & ((b_{1},\ldots,b_{n})+(b_{2},-b_{1},b_{3},\ldots,b_{n-1},-b_{n}),\epsilon_{1}-\epsilon_{2})\\
 & = & 2b_{1}.
\end{eqnarray*}
 $w_{2}\cdot x_{\lambda}=x_{\lambda}$ implies $a_{0}=a_{1}$ and
$a_{n-1}=a_{n}$. From this and the relation \ref{a_i rel D_n odd},
we get that $\tilde{c}(\omega_{1},w_{2})=1$. Also, $(w_{1},x_{1})=1$. 

$w_{3}:\mathbb{R}^{n}\rightarrow\mathbb{R}^{n}$ is the map $(y_{1},\ldots,y_{n})\longmapsto(y_{n},-y_{n-1},\ldots,y_{1})$.
As a product of reflections $w_{3}=s_{n-1}\cdots s_{1}s_{n}\cdots s_{n}s_{4}\cdots s_{n-1}s_{3}\cdots s_{n-2}s_{n}s_{2}\cdots s_{n-1}$.
Using equations \ref{eq:cocycle-1} and \ref{eq: keys coeff} to calculate
$\tilde{c}(\omega_{1},w_{3})$ we get, 
\begin{eqnarray*}
\tilde{c}(\omega_{1},w_{3}) & = & (s_{n}\cdots s_{n}s_{4}\cdots s_{n-1}s_{3}\cdots s_{n-2}s_{n}s_{2}\cdots s_{n-1}\cdot x_{\lambda},\check{\alpha}_{1})\\
 & = & ((b_{1},b_{n},-b_{n-1},-b_{n-2},\ldots,-b_{2}),\epsilon_{1}-\epsilon_{2})\\
 & = & b_{1}-b_{n}\\
 & = & a_{1}+\cdots+a_{n-1}.
\end{eqnarray*}
 Since $w_{3}\cdot x_{\lambda}=x_{\lambda}$ implies $a_{0}=a_{n-1}$
and $a_{1}=a_{n}$. From this and the relation \ref{a_i rel D_n odd},
we get that $\tilde{c}(\omega_{1},w_{3})=\frac{1}{2}$. Also, $(w_{1},x_{n-1})=\frac{1}{2}$. 

The calculation for other hyperspecial vertices follows by the same
method. This completes the case of $D_{n}$ when $n$ is even.

\subsubsection{Type $E_{6}$ }

$E_{6}$ can be realized as a subspace of $\mathbb{R}^{8}$. A set
of simple roots would then be $\Delta=\{\alpha_{1},\ldots,\alpha_{6}\}$
where $\alpha_{1}=(-\frac{1}{2},\frac{1}{2},\frac{1}{2},\frac{1}{2},\frac{1}{2},\frac{1}{2},\frac{1}{2},-\frac{1}{2})$,
$\alpha_{2}=(1,1,0,0,0,0,0,0)$, $\alpha_{3}=(-\frac{1}{2},\frac{1}{2},-\frac{1}{2},-\frac{1}{2},-\frac{1}{2},-\frac{1}{2},-\frac{1}{2},\frac{1}{2})$,
$\alpha_{4}=(0,-1,1,0,0,0,0,0)$, \linebreak{}
$\alpha_{5}=(0,0,-1,1,0,0,0,0)$, $\alpha_{6}=(0,0,0,-1,1,0,0,0)$.
The fundamental weights are,
\begin{eqnarray*}
x_{1} & = & (0,0,0,0,0,-\frac{2}{3},-\frac{2}{3},\frac{2}{3})\\
x_{2} & = & (\frac{1}{2},\frac{1}{2},\frac{1}{2},\frac{1}{2},\frac{1}{2},-\frac{1}{2},-\frac{1}{2},\frac{1}{2})\\
x_{3} & = & (-\frac{1}{2},\frac{1}{2},\frac{1}{2},\frac{1}{2},\frac{1}{2},-\frac{5}{6},-\frac{5}{6},\frac{5}{6})\\
x_{4} & = & (0,0,1,1,1,-1,-1,1)\\
x_{5} & = & (0,0,0,1,1,-\frac{2}{3},-\frac{2}{3},\frac{2}{3})\\
x_{6} & = & (0,0,0,0,1,-\frac{1}{3},-\frac{1}{3},\frac{1}{3}).
\end{eqnarray*}
 Let $\beta$ denote the highest root. Let $\tilde{w_{0}}$ be the
generator of the $\Omega(\check{\Psi})$ and let $w_{0}$ be the vector
part. Then under the action of $w_{0}$ we have $-\beta\mapsto\alpha_{1}\mapsto\alpha_{6}\mapsto-\beta$
and $\alpha_{2}\mapsto\alpha_{3}\mapsto\alpha_{5}\mapsto\alpha_{2}$
. As a product of reflections, $w_{0}=s_{1}s_{3}s_{4}s_{5}s_{6}s_{2}s_{4}s_{5}s_{3}s_{4}s_{1}s_{3}s_{2}s_{4}s_{5}s_{6}$.
Write
\begin{eqnarray*}
x_{\lambda} & = & \sum_{i=1}^{6}a_{i}x_{i}\\
 & = & (b_{1},b_{2},\ldots,b_{8}),
\end{eqnarray*}
 for some $a_{i}\in\mathbb{Q}$ and where $b_{i}$ denote the $i^{th}$
co-ordinate in the standard basis. Let $\omega_{i}$ be the $i^{th}$
fundamental co-weight. Then $\mathcal{C}=\{0,\omega_{1},\omega_{6}\}$.
Using equations \ref{eq:cocycle-1} and \ref{eq: keys coeff} to calculate
$\tilde{c}(\omega_{6},w_{0})$ we get, 
\begin{eqnarray*}
\tilde{c}(\omega_{6},w_{0}) & = & (x_{\lambda}+s_{2}s_{4}s_{5}s_{3}s_{4}s_{1}s_{3}s_{2}s_{4}s_{5}s_{6}\cdot x_{\lambda},\check{\alpha}_{6})\\
 & = & 2b_{5}\\
 & = & a_{2}+a_{3}+2(a_{4}+a_{5}+a_{6}).
\end{eqnarray*}
 From the relation 
\begin{equation}
a_{0}+a_{1}+a_{6}+2(a_{2}+a_{3}+a_{5})+3a_{4}=1
\end{equation}
 and the relations 
\begin{equation}
a_{0}=a_{1}=a_{6},
\end{equation}

\begin{equation}
a_{2}=a_{3}=a_{5},
\end{equation}
 we get that $\tilde{c}(\omega_{6},w_{0})=\frac{2}{3}$. Also $(\omega_{6},x_{1})=\frac{2}{3}$. 

The calculation for the other hyperspecial vertex is very similar.

\subsubsection{Type $E_{7}$}

$E_{7}$ can be realized as a subspace of $\mathbb{R}^{8}$. A set
of simple roots would then be $\Delta=\{\alpha_{1},\ldots,\alpha_{6}\}$
where $\alpha_{1}=(\frac{1}{2},\frac{-1}{2},\frac{-1}{2},-\frac{1}{2},-\frac{1}{2},-\frac{1}{2},-\frac{1}{2},\frac{1}{2})$,
$\alpha_{2}=(1,1,0,0,0,0,0,0)$, $\alpha_{3}=(-1,1,0,0,0,0,0,0)$,
$\alpha_{4}=(0,-1,1,0,0,0,0,0)$, $\alpha_{5}=(0,0,-1,1,0,0,0,0)$,
$\alpha_{6}=(0,0,0,-1,1,0,0,0)$ and $\alpha_{7}=(0,0,0,0,-1,1,0,0)$.
The fundamental weights are,
\begin{eqnarray*}
x_{1} & = & (0,0,0,0,0,0,-1,1),\\
x_{2} & = & (\frac{1}{2},\frac{1}{2},\frac{1}{2},\frac{1}{2},\frac{1}{2},\frac{1}{2},-1,1),\\
x_{3} & = & (-\frac{1}{2},\frac{1}{2},\frac{1}{2},\frac{1}{2},\frac{1}{2},\frac{1}{2},-\frac{3}{2},\frac{3}{2}),\\
x_{4} & = & (0,0,1,1,1,1,-2,2),\\
x_{5} & = & (0,0,0,1,1,1,-\frac{3}{2},\frac{3}{2}),\\
x_{6} & = & (0,0,0,0,1,1,-1,1),\\
x_{7} & = & (0,0,0,0,0,1,-\frac{1}{2},\frac{1}{2}).
\end{eqnarray*}
 Let $\beta$ denote the highest root. Let $\tilde{w_{0}}$ be the
generator of the $\Omega(\check{\Psi})$ and let $w_{0}$ be the vector
part. Then $w_{0}$ is of order two and it sends $-\beta\mapsto\alpha_{7}$,
$\alpha_{1}\mapsto\alpha_{6}$, $\alpha_{3}\mapsto\alpha_{5}$ and
it fixes $\alpha_{2}$ and $\alpha_{4}$. As a product of reflections\linebreak{}
 $w_{0}=s_{7}s_{6}s_{5}s_{4}s_{3}s_{2}s_{1}s_{4}s_{3}s_{5}s_{4}s_{2}s_{6}s_{5}s_{4}s_{3}s_{1}s_{7}s_{6}s_{5}s_{4}s_{2}s_{3}s_{4}s_{5}s_{6}s_{7}$.
Write 
\begin{eqnarray*}
x_{\lambda} & = & \sum_{i=1}^{6}a_{i}x_{i}\\
 & = & (b_{1},b_{2},\ldots,b_{8}),
\end{eqnarray*}
 for some $a_{i}\in\mathbb{Q}$ and where $b_{i}$ denote the $i^{th}$
co-ordinate in the standard basis. Let $\omega_{i}$ be the $i^{th}$
fundamental co-weight. Then $\mathcal{C}=\{0,\omega_{7}\}$. Using
equations \ref{eq:cocycle-1} and \ref{eq: keys coeff} to calculate
$\tilde{c}(\omega_{7},w_{0})$ we get, 
\begin{eqnarray*}
\tilde{c}(\omega_{7},w_{0}) & = & (x_{\lambda}+s_{6}s_{5}s_{4}s_{2}s_{3}s_{4}s_{5}s_{6}s_{7}\cdot x_{\lambda}+\\
 &  & s_{6}s_{5}s_{4}s_{3}s_{2}s_{1}s_{4}s_{3}s_{5}s_{4}s_{2}s_{6}s_{5}s_{4}s_{3}s_{1}s_{7}s_{6}s_{5}s_{4}s_{2}s_{3}s_{4}s_{5}s_{6}s_{7}\cdot x_{\lambda},\check{\alpha}_{7})\\
 & = & 2b_{6}-b_{7}+b_{8}.
\end{eqnarray*}
From the relations 
\begin{equation}
a_{0}+2a_{1}+2a_{2}+3a_{3}+4a_{4}+3a_{5}+2a_{6}+a_{7}=1,
\end{equation}
\begin{equation}
a_{0}=a_{7}\mbox{; }a_{1}=a_{6}\mbox{; }a_{3}=a_{5},
\end{equation}
 we get that $\tilde{c}(\omega_{7},w_{0})=\frac{3}{2}$. Also $(\omega_{7},x_{7})=\frac{3}{2}$.
This completes the case of $E_{7}$.

\subsubsection{Types $^{2}A_{2n-1}(n\geq3)$ and $^{2}D_{n+1}(n\geq2)$ }

The relative root system for $^{2}A_{2n-1}(n\geq3)$ is of type $B_{n}$
and that $^{2}D_{n+1}(n\geq2)$ is of type $C_{n}$. So the calculations
in these cases are the same as in the unramified case. 

This completes the proof of the Proposition \ref{intertwining coeff}
and \ref{intertwining coeff-1}.

\section*{Acknoledgement}

I am indepted to my advisor Jiu-Kang Yu without whose guidance this
work would not be possible. I would like to express my gratitude to
Freydoon Shahidi for many helpful suggestions, particularly for section
\ref{non unitary case}. I am also grateful to David Goldberg for
his careful proof reading and many helpul suggestions. Finally, I
am very thankful to Sandeep Varma for his help and support at various
points of the development of this paper.

\bibliographystyle{alpha}
\bibliography{refrences1}

\end{document}